\documentclass[a4paper, 12pt]{article}

\usepackage{amsmath} 
\usepackage{theorem}
\usepackage{amssymb}
\usepackage{amsfonts}
\usepackage{latexsym}
\usepackage{amscd}
\usepackage{diagramb}
\input GrCalc4.sty
\usepackage{graphicx}
\usepackage{enumerate}
\usepackage{cancel}

\usepackage{pxfonts}

\usepackage{stmaryrd}

\DeclareMathAlphabet{\mathpzc}{OT1}{pzc}{m}{it}

\usepackage{xspace,colortbl}
\usepackage{color}

\definecolor{verde}{rgb}{0.,0.7,0.}
\definecolor{indigo}{rgb}{.18, .34, .78}
\definecolor{indigo1}{rgb}{.18, .24, .78}
\definecolor{indigo2}{rgb}{.18, .14, .78}
\definecolor{indigo3}{rgb}{.18, 0., .78}
\definecolor{rojo}{rgb}{1,0,0}

\definecolor{negro}{rgb}{0,0,0}
\definecolor{grey}{rgb}{0.5,0.5,0.5}

\definecolor{lila}{rgb}{.46, .16, .78}
\definecolor{lila1}{rgb}{.46, .16, .86}
\definecolor{lila2}{rgb}{.56, .16, .86}
	\definecolor{lila3}{rgb}{.63, .16, .78}
\definecolor{lila4}{rgb}{.7, .16, .78}
\definecolor{lila5}{rgb}{.78, .26, .78}
\definecolor{lila6}{rgb}{.6, 0., .78}

\theoremstyle{plain}
\theoremheaderfont{\normalfont\bfseries}

\newtheorem{thm}{Theorem}[section]

\newtheorem{lma}[thm]{Lemma}
\newtheorem{cor}[thm]{Corollary}
\newtheorem{defn}[thm]{Definition}

\newtheorem{prop}[thm]{Proposition}

\theorembodyfont{\rmfamily}

\newtheorem{rem}[thm]{Remark}
\newtheorem{ex}[thm]{Example}
\newcommand{\qed}{\hfill\quad\fbox{\rule[0mm]{0,0cm}{0,0mm}}  \par\bigskip}

\newcommand{\w}{\hspace{-0,06cm}}
\newcommand{\s}{\hspace{0,06cm}}
\newcommand{\Cat}{\operatorname {Cat}}

\newcommand{\Mnd}{{\rm Mnd}}
\newcommand{\Comnd}{{\rm Comnd}}
\newcommand{\Bimnd}{{\rm Bimnd}}

\newcommand{\BQ}{{\rm BQ}}

\def\dul#1{\underline{\underline{#1}}}
\def\dcr#1{\crta{\crta{#1}}}

\newcommand{\comp}{\circ}
\newcommand{\iso}{\cong}
\newcommand{\ot}{\otimes}

\newcommand{\C}{{\mathcal C}}
\newcommand{\Tau}{{\mathcal T}}
\newcommand{\I}{{\mathcal I}}

\newcommand{\D}{{\mathcal D}}
\newcommand{\F}{{\mathcal F}}
\newcommand{\G}{{\mathcal G}}

\newcommand{\A}{{\mathcal A}}
\newcommand{\B}{{\mathcal B}}
\newcommand{\E}{{\mathcal E}}

\newcommand{\Ll}{{\mathcal L}}
\newcommand{\Pp}{{\mathcal P}}

\newcommand{\U}{{\mathcal U}}
\newcommand{\YD}{{\mathcal YD}}

\newcommand{\crta}{\overline}

\newcommand{\Fi}{\varphi}
\newcommand{\Id}{\operatorname {Id}}
\newcommand{\id}{\operatorname {id}}

\newcommand{\Epsilon}{\varepsilon}
\newcommand{\End}{\operatorname {End}}
\newcommand{\Aut}{\operatorname {Aut}}

\def\K{{\mathcal K}}  

\def\u#1{\underline{#1}}

\newcommand{\cref}[1]{C.~\ref{c:#1}}

\newcommand{\exlabel}[1]{\label{ex:#1}}
\newcommand{\exref}[1]{Example~\ref{ex:#1}}
\newcommand{\lelabel}[1]{\label{le:#1}}
\newcommand{\leref}[1]{Lemma~\ref{le:#1}}
\newcommand{\eqlabel}[1]{\label{eq:#1}}
\newcommand{\equref}[1]{(\ref{eq:#1})}

\newcommand{\delabel}[1]{\label{de:#1}}
\newcommand{\deref}[1]{Definition~\ref{de:#1}}
\newcommand{\prlabel}[1]{\label{pr:#1}}
\newcommand{\prref}[1]{Proposition~\ref{pr:#1}}

\newcommand{\selabel}[1]{\label{se:#1}}
\newcommand{\seref}[1]{Section~\ref{se:#1}}

\begin{document}

\title{ Turaev bicategories, \\ 
generalized Yetter-Drinfel`d \\ modules in 2-categories and \\ 
a Turaev 2-category for bimonads \\ in 2-categories} 
\author{Bojana Femi\'c \vspace{6pt} \\
{\small Facultad de Ingenier\'ia, \vspace{-2pt}}\\
{\small  Universidad de la Rep\'ublica} \vspace{-2pt}\\
{\small  Julio Herrera y Reissig 565,} \vspace{-2pt}\\
{\small  11 300 Montevideo, Uruguay}}

\date{}

\maketitle

\begin{abstract}
We introduce Turaev bicategories and Turaev pseudofunctors. On the one hand, they generalize the notions of Turaev categories (and Turaev functors), 
introduced at the turn of the millennium and originally called ``crossed group categories'' by Turaev himself, and the notions of bicategories 
and pseudofunctors, on the other. For bimonads in 2-categories, which we defined in one of our previous papers, we introduce generalized Yetter-Drinfel`d 
modules in 2-categories. These generalize to the 2-categorical setting the generalized Yetter-Drinfel`d modules (over a field) of Panaite and Staic, 
and thus also in particular the anti Yetter-Drinfel`d modules, introduced by Hajac-Khalkhali-Rangipour-Sommerhauser as coefficients for the cyclic 
cohomology of Hopf algebras, defined by Connes and Moscovici. We construct Turaev 2-category for bimonads in 2-categories as a Turaev extension 
of the 2-category of bimonads. This Turaev 2-category generalizes the Turaev category of generalized Yetter-Drinfel`d modules of Panaite and Staic. 
We also prove in the 2-categorical setting their results on pairs in involution, which in turn go back to modular pairs in involution of 
Connes and Moscovici.

\bigbreak

\medskip

\end{abstract}

\section{Introduction}

In \cite{Tur99,Tur00} Turaev introduced 2- and respectively 3-dimensional homotopy quantum field theory (HQFT). It is a version of a topological quantum field theory (TQFT) 
for manifolds $M$ equipped with homotopy classes of maps $M\to K(G,1)$. Here $K(G,1)$ is the Eilenberg-MacLane space determined by a group $G$ and the latter homotopy classes 
of maps classify principal $G$-bundles over $M$. For the purposes of the 3-dimensional case Turaev introduced the notion of {\em crossed group categories} and 
{\em modular crossed group categories} which yield a 3-dimensional HQFT with target $K(G,1)$. When the group $G$ is trivial, one recovers the usual construction of 
3-dimensional TQFT from modular categories. 

Since their introduction crossed group categories have been studied in different algebraic contexts in the works of M. Zunino, M. Lombaerde and S. Caenepeel, F. Panaite and M.D. Staic, 
S. Wang, to mention some of them. 
In Turaev's definition a crossed group category is a $k$-additive rigid monoidal category which is a disjoint union of categories indexed by a group $G$ and satisfies 
certain conditions. Zunino studied {\em Turaev categories} in \cite{Zun} generalizing Turaev's definition to categories which are just monoidal, skipping the additivity 
and rigidity condition. When the group $G$ is trivial, one recovers the definition of a monoidal category. Working in this setting, in \cite{PS} the authors 
introduced {\em generalized Yetter-Drinfel`d modules} over a Hopf algebra $H$ and showed that they form a braided Turaev category. The Yetter-Drinfel`d condition in a 
generalized Yetter-Drinfel`d module is twisted by two elements of the group, say $(\alpha,\beta)$ with $\alpha, \beta\in G$. If $\alpha$ and $\beta$ are trivial, one recovers 
the usual Yetter-Drinfel`d condition. When $\alpha=S^2$ and $\beta=\id_H$, being $S$ the antipode of $H$, one recovers the anti-Yetter-Drinfel`d modules introduced in \cite{H1,H2} as coefficients for the cyclic cohomology of Hopf algebras, defined by Connes and Moscovici in \cite{CM}.  
An $(\id_H,\beta)$-Yetter-Drinfel`d module appeared in \cite{CVZ}, where the authors constructed a group anti-homomorphism $\Aut(H)\to\BQ(k;H)$ from the Hopf automorphism group of a 
finite-dimensional Hopf algebra to the quantum Brauer group of $H$. We generalized this group anti-homomorphism to braided monoidal categories in \cite{F2}. 

\smallskip

In \cite{Femic5} we introduced bimonads in 2-categories and the 2-category $\Bimnd(\K)$ of bimonads in a 2-category $\K$. 
We showed that given a braided monoidal category $\C$, a bimonad in the induced 2-category $\dul{\C}$ is a bialgebra in $\C$ and that 1-cells over the same 0-cell 
$F$ in $\Bimnd(\K)$ are Yetter-Drinfel`d modules over the bialgebra $F$ in $\C$. Moreover, the composition of 1-cells coincides with the 
tensor product in the monoidal category of Yetter-Drinfel`d modules over a bialgebra $F$ in $\C$. This lead us to define Yetter-Drinfel`d modules in 2-categories. 
More precisely, in \cite{Femic7} we call the 1-cells in $\Bimnd(\K)$ {\em strong Yetter-Drinfel`d modules in $\K$} which imply the usual Yetter-Drinfel`d condition. 

\medskip

In the current work, inspired by \cite{PS} we generalize the construction from therein to the 2-categorical setting, obtaining ``generalized 
strong Yetter-Drinfel`d modules in $\K$''. We introduce {\em Turaev bicategories} and {\em Turaev pseudofunctors}. 
Turaev bicategory generalizes the notion of a Turaev category, on one hand, and the notion of a bicategory, on the other. 
Indeed, a Turaev bicategory $\Ll$ consists of a class of 0-cells $\Ll^0$ so that for every two 0-cells $\A,\B\in\Ll^0$ there is a group 
$G_{\A,\B}$ so that the category $\Ll(\A,\B)=\displaystyle{\bigcup_{\alpha\in G_{\A,\B}}^{\bullet} } \Ll(\A, \B)_{\alpha}$ 
is a disjoint union of categories $\Ll(\A, \B)_{\alpha}$. The definition is such that for every $\A$ the category $\Ll(\A,\A)$ is a Turaev category. 
On the other hand, if for all $\A,\B\in\Ll^0$ the group $G_{\A,\B}$ is trivial, then $\Ll$ is a bicategory. Analogously, Turaev pseudofunctors 
generalize both Turaev functors and pseudofunctors. 

Under certain conditions Turaev bicategories have an underlying bicategory, in which case the Turaev bicategory is said to be a Turaev extension 
of the latter. 
We construct a Turaev 2-category for bimonads in $\K$, we denote it by $\Bimnd^T(\K)$. Indeed, it is a {\em Turaev extension} of the 2-category 
$\Bimnd(\K)$ of bimonads that we studied in \cite{Femic7}. Its 1-cells over the same 0-cell are generalized Yetter-Drinfel`d modules in $\K$. 
This Turaev 2-category generalizes the Turaev category of generalized Yetter-Drinfel`d modules from \cite{PS}:  
if $\Ll$ is the 2-category $\dul{Vec}$ induced by the braided monoidal category of vector spaces - where the unique 0-cell is a singleton $*$ - then 
the category $\dul{Vec}(*,*)$ is the Turaev category of Panaite and Staic. 

In our Turaev 2-category for bimonads in $\K$ for every  0-cell $\A$ there is a group $G(\A)$ and for two 0-cells $\A$ and $\B$ the 
group $G_{\A,\B}$ is given by the Cartesian product $G(\B)\times G(\A)$. We introduce {\em pairs in involution} corresponding to pairs 
$(\alpha, \beta)\in G(\B)\times G(\A)$, when the groups $G(\B)$ and $G(\A)$ are isomorphic. We prove that in this case the categories 
$\Bimnd^T(\K)(\A,\B)_{(\alpha, \beta)}$ and $\Bimnd(\K)(\A,\B)$ are isomorphic. This generalizes the corresponding result in \cite{PS}. 
Our pairs in involution are a 2-categorification of {\em modular pairs in involution} introduced by Connes and Moscovici. 

\medskip

The composition of the paper is the following. In the next Section we recall the definition of Turaev category and our definition of the 2-category 
$\Bimnd(\K)$ of bimonads in a 2-category $\K$. 
In the third one we introduce {\em transitive system of groups with projections} which is a part of the data of the definition of a Turaev bicategory, 
and we define Turaev bicategories and Tuarev pseudofunctors. In Section 4 we construct a transitive system of groups for bimonads in $\K$ and 
a Turaev 2-category for bimonads in $\K$. This construction is completed in Subsection 4.3, which ends with some examples. In the last Subsection of the 
paper we study pairs in involution for bimonads in $\K$.

\section{Preliminaries} \selabel{prelim}

We start the preliminary Section by recalling the definition of a Turaev category. We consider the definition from \cite{Zun}, which slightly generates 
the original Turaev's definition of a ``crossed group category'' from \cite{Tur00}. A Turaev category $\Tau$ over a group $G$ is a triple $(\Tau, G, \Fi)$, where 
$\Tau$ is a monoidal category which is a disjoint union of a family of subcategories $\{\Tau_\alpha \s\vert\s \alpha\in G\}$, such that for all $\alpha,\beta\in G$ 
and $X\in\Tau_\alpha, Y\in\Tau_\beta$ the tensor product $X\ot Y\in\Tau_{\alpha\beta}$, and $\Fi:G\to\Aut_0(\Tau), \beta\mapsto\Fi_\beta$ is a group 
homomorphism to the group $\Aut_0(\Tau)$ of strict monoidal automorphism functors of $\Tau$, such that 
$\Fi_\beta(\Tau_\alpha)=\Tau_{\beta\alpha\beta^{-1}}$, for all $\alpha, \beta$.


For the basics on 2-categories and 2-(co)monads we recommend \cite{Be, Bo, Neu, St1}. We fix the following notation. 
The arrows for 2-cells we denote the same way as 1-cells, the distinction will be clear from the context. The horizontal composition of 2-cells 
we denote by $\times$ and the vertical one by $\comp$. Throughout $\K$ will denote a 2-category. The identities between 2-cells in $\K$ acting on 
composable 1-cells we will express in string diagrams, whose source and target objects represent the 1-cells in $\K$. The horizontal juxtaposition 
in string diagrams corresponds to the horizontal composition of 2-cells, while the vertical juxtaposition corresponds to the vertical composition 
of 2-cells. This makes the use of string diagrams natural for computations in 2-categories, as it is for monoidal categories $\C$. 
Multiplication and unit of a monad (or an algebra in $\C$), commultiplication and counit of a comonad (or a coalgebra in $\C$), left action of an algebra and 
a left coaction of a coalgebra in $\C$ we write respectively: 
$$
\mu=\gbeg{2}{1}
\gmu \gnl
\gend \qquad 
\eta=\gbeg{1}{1}
\gu{1} \gnl
\gend \qquad 
\Delta=\gbeg{2}{1}
\gcmu \gnl
\gend \qquad 
\Epsilon=\gbeg{1}{1}
\gcu{1} \gnl
\gend  \qquad 
\gbeg{2}{1}
\glm \gnl
\gend  \qquad 
\gbeg{2}{1}
\glcm \gnl
\gend.
$$

\medskip

We introduced bimonads in $\K$ and their 2-category $\Bimnd(\K)$ in \cite{Femic5}. We will work here with a version of the 2-category $\Bimnd(\K)$ 
differing from the latter in 2-cells, as we did in \cite{Femic7}. We recall the necessary definitions here. 

A {\em bimonad} in $\K$ is a quintuple $(\A, F, \mu, \eta, \Delta, \Epsilon, \lambda)$ where $(\A, F, \mu, \eta)$ is a monad and $(\A, F, \Delta, \Epsilon)$ is a comonad so that 
the following compatibility conditions hold: 
\begin{equation} \eqlabel{bimonad}
\gbeg{3}{5}
\got{1}{F} \got{3}{F} \gnl
\gwmu{3} \gnl
\gvac{1} \gcl{1} \gnl
\gwcm{3} \gnl
\gob{1}{F}\gvac{1}\gob{1}{F}
\gend=
\gbeg{4}{5}
\got{1}{F} \got{3}{F} \gnl
\gcl{1} \gwcm{3} \gnl
\glmptb \gnot{\hspace{-0,34cm}\lambda} \grmptb \gvac{1} \gcl{1} \gnl
\gcl{1} \gwmu{3} \gnl
\gob{1}{F} \gvac{1} \gob{2}{\hspace{-0,34cm}F}
\gend, \quad
\gbeg{2}{3}
\got{1}{F} \got{1}{F} \gnl
\gcl{1} \gcl{1} \gnl
\gcu{1}  \gcu{1} \gnl
\gend=
\gbeg{2}{3}
\got{1}{F} \got{1}{F} \gnl
\gmu \gnl
\gvac{1} \hspace{-0,2cm} \gcu{1} \gnl
\gob{1}{}
\gend, \quad
\gbeg{2}{3}
\gu{1}  \gu{1} \gnl
\gcl{1} \gcl{1} \gnl
\gob{1}{F} \gob{1}{F}
\gend=
\gbeg{2}{3}
\gu{1} \gnl
\hspace{-0,34cm} \gcmu \gnl
\gob{1}{F} \gob{1}{F}
\gend, \quad
\gbeg{1}{2}
\gu{1} \gnl
\gcu{1} \gnl
\gob{1}{}
\gend=
\Id_{id_{\A}}
\end{equation}
and the 2-cell $\lambda: FF\to FF$ is such that $(F, \lambda)$ is a 1-cell both in $\Mnd(\K)$ and in $\Comnd(\K)$ (see \equref{monadic d.l.} and 
\equref{comonadic d.l.} below with $F\s'=X=F$). We will write shortly for a bimonad: $(\A, F, \lambda)$, or just $(F, \lambda)$ or $(\A,F)$. 

\medskip

The 2-category $\Bimnd(\K)$ of bimonads in $\K$ has bimonads for 0-cells, 1-cells are triples $(X,\psi,\phi)$ 
where $(X, \psi)$ is a 1-cell in $\Mnd(\K)$, $(X, \phi)$ is a 1-cell in $\Comnd(\K)$, meaning that the identities 
\vspace{-1,4cm}
\begin{center} \hspace{-0,6cm}
\begin{tabular}{p{7.4cm}p{0cm}p{8cm}}
\begin{equation}\eqlabel{monadic d.l.}
\gbeg{3}{5}
\got{1}{F'}\got{1}{F'}\got{1}{X}\gnl
\gcl{1} \glmpt \gnot{\hspace{-0,34cm}\psi} \grmptb \gnl
\glmptb \gnot{\hspace{-0,34cm}\psi} \grmptb \gcl{1} \gnl
\gcl{1} \gmu \gnl
\gob{1}{X} \gob{2}{F}
\gend=
\gbeg{3}{5}
\got{1}{F'}\got{1}{F'}\got{1}{X}\gnl
\gmu \gcn{1}{1}{1}{0} \gnl
\gvac{1} \hspace{-0,34cm} \glmptb \gnot{\hspace{-0,34cm}\psi} \grmptb  \gnl
\gvac{1} \gcl{1} \gcl{1} \gnl
\gvac{1} \gob{1}{X} \gob{1}{F}
\gend;
\quad
\gbeg{2}{5}
\got{3}{X} \gnl
\gu{1} \gcl{1} \gnl
\glmptb \gnot{\hspace{-0,34cm}\psi} \grmptb \gnl
\gcl{1} \gcl{1} \gnl
\gob{1}{X} \gob{1}{F}
\gend=
\gbeg{3}{5}
\got{1}{X} \gnl
\gcl{1} \gu{1} \gnl
\gcl{2} \gcl{2} \gnl
\gob{1}{X} \gob{1}{F}
\gend
\end{equation} & &
 \begin{equation}\eqlabel{comonadic d.l.}
\gbeg{3}{5}
\got{1}{X} \got{2}{F}\gnl
\gcl{1} \gcmu \gnl
\glmptb \gnot{\hspace{-0,34cm}\phi} \grmptb \gcl{1} \gnl
\gcl{1} \glmptb \gnot{\hspace{-0,34cm}\phi} \grmptb \gnl
\gob{1}{F\s'} \gob{1}{F\s'} \gob{1}{X}
\gend=
\gbeg{3}{5}
\got{2}{X} \got{1}{\hspace{-0,2cm}F}\gnl
\gcn{1}{1}{2}{2} \gcn{1}{1}{2}{2} \gnl
\gvac{1} \hspace{-0,34cm} \glmpt \gnot{\hspace{-0,34cm}\phi} \grmptb \gnl
\gvac{1} \hspace{-0,2cm} \gcmu \gcn{1}{1}{0}{1} \gnl
\gvac{1} \gob{1}{F\s'} \gob{1}{F\s'} \gob{1}{X}
\gend;
\quad
\gbeg{3}{5}
\got{1}{X} \got{1}{F} \gnl
\gcl{1} \gcl{1} \gnl
\glmptb \gnot{\hspace{-0,34cm}\phi} \grmptb \gnl
\gcu{1} \gcl{1} \gnl
\gob{3}{X}
\gend=
\gbeg{3}{5}
\got{1}{X} \got{1}{F} \gnl
\gcl{1} \gcl{1} \gnl
\gcl{2}  \gcu{1} \gnl
\gob{1}{X}
\gend
\end{equation}
\end{tabular}
\end{center} 
hold, and moreover the compatibility 
\begin{equation}  \eqlabel{psi-lambda-phi 2-bimonad}
\gbeg{3}{5}
\got{1}{F'} \got{1}{X} \got{1}{F} \gnl
\glmptb \gnot{\hspace{-0,34cm}\psi} \grmptb \gcl{1} \gnl
\gcl{1} \glmptb \gnot{\hspace{-0,34cm}\lambda} \grmptb \gnl
\glmptb \gnot{\hspace{-0,34cm}\phi} \grmptb \gcl{1} \gnl
\gob{1}{F'} \gob{1}{X} \gob{1}{F}
\gend=
\gbeg{3}{5}
\got{1}{F'} \got{1}{X} \got{1}{F} \gnl
\gcl{1} \glmptb \gnot{\hspace{-0,34cm}\phi} \grmptb \gnl
\glmptb \gnot{\hspace{-0,34cm}\lambda'} \grmptb \gcl{1} \gnl
\gcl{1} \glmptb \gnot{\hspace{-0,34cm}\psi} \grmptb \gnl
\gob{1}{F'} \gob{1}{X} \gob{1}{F}
\gend
\end{equation} 
is fulfilled. The 2-cells of $\Bimnd(\K)$ are 2-cells both in $\Mnd(\K)$ and $\Comnd(\K)$ simultaneously. The composition of the 2-cells is defined in the obvious way, 
the identity 1-cell on a bimonad $(\A,F)$ is given by $(\Id_\A, \id_F)$, and the identity 2-cell on a 1-cell $(X,\psi,\phi)$ is given by $\id_X$.

\section{Turaev bicategories}

In this Section we define Turaev bicategories and Turaev pseudofunctors. A constituting part of the data for the former is a transitive system of groups which we introduce first.

\subsection{Transitive system of groups with projections}

Let $\I$ be an index class. By a {\em transitive system of groups with projections over $\I$} we mean a family of groups $\{G_{\A,\B}\s\vert\s \A,\B\in\I\}$ so that for 
all $\A,\B,\C\in\I$ the following is fulfilled:
\begin{enumerate}
\item there is a {\em transitive product}
$$\hspace{-1,4cm} *^{\C,\B,\A}: G_{\B,\C}\times G_{\A,\B}\to G_{\A,\C}$$
$$\quad(\alpha,\beta)\quad\mapsto\quad\alpha*\beta$$
which is associative, meaning that given a 
fourth index $\D$ and a group $G_{\C,\D}$ the following two compositions of maps are equal:
$$(G_{\C,\D}\times G_{\B,\C})\times G_{\A,\B} \stackrel{*^{\D,\C,\B}\times\id}{\longrightarrow} G_{\B,\D}\times G_{\A,\B} \stackrel{*^{\D,\B,\A}}{\longrightarrow} G_{\A,\D}$$
and:
$$G_{\C,\D}\times (G_{\B,\C}\times G_{\A,\B}) \stackrel{\id\times *^{\C,\B,\A}}{\longrightarrow} G_{\C,\D}\times G_{\A,\C} \stackrel{*^{\D,\C,\A}}{\longrightarrow} G_{\A,\D},$$
and it holds $e_{\B,\B}*\alpha=\alpha=\alpha* e_{\A,\A}$ and $e_{\B,\C}*e_{\A,\B}=e_{\A,\C}$ for all $\alpha\in G_{\A,\B}$, here $e_{-,-}$ denotes the corresponding units, 
and if $\A=\B=\C$ the transitive product $*^{\C,\B,\A}$ coincides with the product in $G_{\A,\A}$;
\item there are {\em projection group maps} $\pi_{12}: G_{\A,\C}\to G_{\B,\C}$ and $\pi_{23}: G_{\A,\C}\to G_{\A,\B}$ yielding 
$$\pi^{\C,\B,\A}: G_{\A,\C}\to G_{\B,\C}\times G_{\A,\B}$$
$$\omega\mapsto (\pi_{12}(\omega), \pi_{23}(\omega))=(\omega_{(1)}, \omega_{(1)})$$
so that: 
\begin{itemize}
 \item the following compatibility between the products in the groups $G_{-,-}$ and the transitive product holds: 
\begin{equation} \eqlabel{conjug compat}
\omega(\beta*\alpha)\omega^{-1}=
(\omega_{(1)}\beta)*(\alpha\omega_{(2)}^{-1})
\end{equation}
for all $\alpha\in G_{\A,\B}, \beta\in G_{\B,\C}$ and $\omega\in G_{\A,\C}$;
\item $\pi^{\C,\B,\A}$ is {\em coassociative}, meaning that given a fourth index $\D$ and a group $G_{\C,\D}$ we have for all $\omega\in G_{\A,\D}$: 
\begin{equation} \eqlabel{coass}
(\omega_{(1)_{(1)}}, \omega_{(1)_{(2)}}, \omega_{(2)}) = (\omega_{(1)}, \omega_{(2)_{(1)}}, \omega_{(2)_{(2)}})
\end{equation}
(we use Sweedler-type notation);
\item $\pi^{\C,\B,\A}$ is {\em counital}: $*(\id\times\Epsilon)\pi^{\C,\B,\A}=\id_{G_{\A,\C}}=*(\Epsilon\times\id)\pi^{\C,\B,\A}$, 
here $\Epsilon$ denotes the trivial group map sending all to the unit element $e$; the former can be written as: 
\begin{equation} \eqlabel{counit law}
\omega_{(1)}* e=\omega = e*\omega_{(2)}
\end{equation}
for all $\omega\in G_{\A,\C}$. 
\end{itemize}
\end{enumerate}

\bigskip

A transitive system of groups with projections as above we denote by $(\{G_{\A,\B}\s\vert\s \A,\B\in\I\}, *, \pi)$. 
From the above definition we have: 
\begin{equation} \eqlabel{antipode rule}
\omega_{(1)}^{-1}*\omega_{(2)}=e_{G_{\A,\C}}=\omega_{(1)}\omega_{(2)}^{-1};
\end{equation}
if $\A=\B=\C$ the group map $\pi: G\to G\times G$ is given by $\pi(g)=(g,g), \forall g\in G$; \\
if $\A=\B$ then $\pi_{12}=\id_{G_{\B,\C}}$, and if $\B=\C$ then $\pi_{23}=\id_{G_{\A,\B}}$.

\subsection{Turaev bicategories and Turaev pseudofunctors}



We will denote by $\Aut(\C)$ the group of automorphisms of a category $\C$. 

\begin{defn} \delabel{2-Tur}
A Turaev bicategory consist of the following data: 
\begin{enumerate}
\item a class of objects {\em i.e.} 0-cells $\Ll^o$;
\item a transitive system of groups with projections over $\Ll^o$: $\left( \{G_{\A,\B}\s\vert\s \A,\B\in\Ll^o\}, *, \pi \right)$; 
\item for any pair of objects $\A, \B\in\Ll^o$:
\begin{enumerate} [i)]
\item a category $\Ll(\A, \B)=\displaystyle{\bigcup_{\alpha\in G_{\A,\B}}^{\bullet} } \Ll(\A, \B)_{\alpha}$ 
which is a disjoint union of categories $\Ll(\A, \B)_{\alpha}$, and 
\item a group map $\Fi^{\A,\B}: G_{\A,\B}\to\Aut(\Ll(\A, \B)), \beta\mapsto\Fi^{\A,\B}_\beta$ such that 
$\Fi^{\A,\B}_\beta(\Ll(\A, \B)_\alpha)=\Ll(\A, \B)_{\beta\alpha\beta^{-1}}$ for every $\alpha, \beta\in G_{\A,\B}$;
\end{enumerate}
\item for all $\A\in\Ll^o$ there is a 1-cell $\Id_\A\in\Ll(\A,\A)_e$, where $e$ is the unit element of $G_{\A,\A}$;
\item for every pair of objects $\A, \B\in\Ll^o$ and every 1-cell $X$ in $\Ll(\A, \B)_\alpha$ there is a 2-cell $\id_X:X\to X$;
\item for all $\A,\B,\C\in\Ll^o$:
\begin{enumerate} [i)]
\item for all $\alpha\in G_{\A,\B}, \beta\in G_{\B,\C}$ there are functors 
$$c^{\beta,\alpha}_{\C,\B,\A}: \Ll(\B, \C)_\beta\times\Ll(\A, \B)_\alpha\to\Ll(\A, \C)_{\beta*\alpha}$$
$$\hspace{2,1cm} (Y,X)\hspace{1,3cm}\mapsto\hspace{0,6cm} Y\cdot X$$
\item for all $\omega\in G_{\A,\C}$ and $Y\in\Ll(\B, \C)_\beta, X\in\Ll(\A, \B)_\alpha$ there are isomorphisms (2-cells) natural in $X$ and $Y$ 
$$s^{\C,\B,\A}_{Y,X}: \Fi^{\A,\C}_\omega(Y\cdot X) \longrightarrow 
\Fi^{\B,\C}_{\omega_{(1)}}(Y) \cdot \Fi^{\A,\B}_{\omega_{(2)}}(X)$$
and a bijective 2-cell: 
$$s^\A_0: \Fi^{\A,\A}_\gamma(\Id_\A)\to\Id_\A$$
for all $\gamma\in G_{\A,\A}$; 
\end{enumerate} 
\item for all $\A,\B,\C,\D\in\Ll^o$ and all $\alpha\in G_{\A,\B}, \beta\in G_{\B,\C},\gamma\in G_{\C,\D}$:
 \begin{enumerate} [i)]
  \item  there are isomorphisms (2-cells) natural in composable 1-cells 
$Z\in\Ll(\C, \D)_\gamma, Y\in\Ll(\B, \C)_\beta, X\in\Ll(\A, \B)_\alpha$ defining the associativity law 
$$a^{\gamma,\beta,\alpha}_{Z,Y,X}: (Z\cdot Y)\cdot X\to Z\cdot (Y\cdot X),$$
and for each  $\A,\B\in\Ll^o$ and $\alpha\in G_{\A,\B}$ there are isomorphisms (2-cells) natural in $X\in\Ll(\A, \B)_\alpha$ defining the left and right unity laws 
$$\lambda_X: \Id_\B\cdot X\to X, \hspace{2cm} \rho_X: X\cdot\Id_\A\to X$$
such that the following pentagonal and triangular diagrams commute for all composable 1-cells 
$W\in\Ll(\D, \E)_\delta, Z\in\Ll(\C, \D)_\gamma, Y\in\Ll(\B, \C)_\beta, X\in\Ll(\A, \B)_\alpha$ and all $\alpha\in G_{\A,\B}, \beta\in G_{\B,\C},\gamma\in G_{\C,\D}, \delta\in G_{\D,\E}$: 
\begin{equation*}
\scalebox{0.84}{
\bfig
\putmorphism(-200,500)(1,0)[((W\cdot Z)\cdot Y)\cdot X` (W\cdot (Z\cdot Y))\cdot X ` a^{\delta,\gamma,\beta}_{W,Z,Y}\times\id]{1360}1a
\putmorphism(1160,500)(1,0)[\phantom{((W\cdot Z)\cdot Y)\cdot X}` W((ZY)X) ` a^{\delta,\gamma\beta,\alpha}_{W,ZY,X}]{1200}1a
\putmorphism(2370,500)(0,-1)[``\id\times a^{\gamma,\beta,\alpha}_{Z,Y,X}]{500}1r
\putmorphism(-160,500)(0,-1)[``a^{\delta\gamma,\beta,\alpha}_{WZ,Y,X}]{500}1l
\putmorphism(-200,0)(1,0)[(W\cdot Z)\cdot (Y\cdot X)` W\cdot (Z\cdot (Y\cdot X))` a^{\delta,\gamma,\beta\alpha}_{W,Z,YX}]{2650}1b
\efig}
\end{equation*}

$$\scalebox{0.84}{
\bfig
\putmorphism(-80,500)(1,0)[(Y\cdot\Id_\B)\cdot X ` Y\cdot (\Id_\B\cdot X)` a_{Y,\Id_\B,X}]{1080}1a
\putmorphism(20,500)(1,-1)[`YX.`\rho_Y\times Y]{430}1l
\putmorphism(930,500)(-1,-1)[``Y \times\lambda_X]{430}1r
\efig}
$$
\item 
for all $Z\in\Ll(\C, \D)_\gamma, Y\in\Ll(\B, \C)_\beta, X\in\Ll(\A, \B)_\alpha$ the following hexagon and two triangles commute:
$$\hspace{-0,6cm}
\scalebox{0.84}{
\bfig
\putmorphism(0,500)(1,0)[ (\Fi^{\C, \D}_{\omega_{(1)_{(1)}}}(Z)\cdot\Fi^{\B, \C}_{\omega_{(1)_{(2)}}}(Y))\cdot\Fi^{\A, \B}_{\omega_{(2)}}(X)  
    `\Fi^{\B, \D}_{\omega_{(1)}}(Z\cdot Y)\cdot\Fi^{\A, \B}_{\omega_{(2)}}(X) ` s^{\D,\B,\A}_{Z,Y}\times\id]{1620}{-1}a
\putmorphism(1720,500)(1,0)[\phantom{\Fi^{\B, \D}_{\omega_{{(1)}}(Z\cdot Y)\cdot\Fi^{\A, \B}_{\omega_{(2)}}(X)}} ` \Fi^{\A, \D}_\omega((Z\cdot Y)\cdot X) ` 
    s_{ZY,X}^{\D,\B,\A}]{1000}{-1}a
\putmorphism(2770,500)(0,-1)[` ` \Fi^{\A, \D}_\omega( a_{Z,Y,X})]{500}1l
\putmorphism(-160,500)(0,-1)[``a_{\Fi(Z),\Fi(Y),\Fi(X)}]{500}1l
\putmorphism(0,0)(1,0)[\Fi^{\C, \D}_{\omega_{{(1)}}}(Z)\cdot(\Fi^{\B, \C}_{\omega_{(2)_{(1)}}}(Y)\cdot\Fi^{\A, \B}_{\omega_{(2)_{(2)}}}(X)) `  
   \Fi^{\C, \D}_{\omega_{(1)}}(Z)\cdot\Fi^{\A, \C}_{\omega_{(2)}}(Y\cdot X) ` \id\times s_{Y,X}^{\C,\B,\A}]{1620}{-1}b
\putmorphism(1620,0)(1,0)[\phantom{\Fi^{\C, \D}_{\omega_{(1)}}(Z)\comp\Fi^{\A, \C}_{\omega_{(2)}}(YX)} ` \Fi^{\A, \D}_\omega(Z\cdot (Y\cdot X)) ` s_{Z,YX}^{\D,\C,\A}]{1260}{-1}b
\efig}
$$
where $\omega\in G_{\A,\D}$ and: 
$$\scalebox{0.84}{
\bfig
\putmorphism(-80,500)(1,0)[\Id_\B\cdot\Fi^{\A,\B}_{\omega_{(2)}}(X) ` \Fi^{\B,\B}_{\omega_{(1)}}(\Id_\B)\cdot\Fi^{\A,\B}_{\omega_{(2)}}(X) ` s_0^{\B}\times\id]{1080}{-1}a
\putmorphism(1000,500)(1,0)[\phantom{ \Fi^{\B,\B}(\Id_\B)\cdot\Fi^{\A,\B}(X)} ` \Fi^{\A,\B}_\omega(\Id_\B\cdot X) ` s_{\Id_\B,X}^{\B,\B,\A}]{1080}{-1}a
\putmorphism(20,500)(2,-1)[`\Fi^{\A,\B}_\omega(X)`]{930}1l
\putmorphism(-20,500)(2,-1)[``\lambda_{\Fi(X)}]{930}0l
\putmorphism(2060,500)(-2,-1)[``\Fi^{\A,\B}_\omega(\lambda)]{1030}1r
\efig}
$$
$$\scalebox{0.84}{
\bfig
\putmorphism(-80,500)(1,0)[\Fi^{\A,\B}_{\omega_{(1)}}(X)\cdot\Id_\A ` \Fi^{\A,\B}_{\omega_{(1)}}(X)\cdot\Fi^{\A,\A}_{\omega_{(2)}}(\Id_\A) ` \id\times s_0^{\A}]{1080}{-1}a
\putmorphism(1000,500)(1,0)[\phantom{ \Fi^{\B,\B}(\Id_\B)\cdot\Fi^{\A,\B}(X)} ` \Fi^{\A,\B}_\omega(X\cdot\Id_\A) ` s_{X,\Id_\A}^{\B,\A,\A}]{1080}{-1}a
\putmorphism(20,500)(2,-1)[`\Fi^{\A,\B}_\omega(X).`]{930}1l
\putmorphism(-50,500)(2,-1)[``\rho_{\Fi(X)}]{930}0l
\putmorphism(2060,500)(-2,-1)[``\Fi^{\A,\B}_\omega(\rho)]{1030}1r
\efig}
$$
\end{enumerate}
where $\omega\in G_{\A,\B}$. 
\end{enumerate}

A Turaev bicategory is said to be {\em Turaev 2-category} if the isomorphisms $s^{\C,\B,\A}_{Y,X}, s^\A_0, a^{\gamma,\beta,\alpha}_{Z,Y,X}, 
\lambda_X$ and $\rho_X$ are identities (in this case the five diagrams in the point 7. above trivially commute). 
\end{defn}


In a Turaev bicategory for every 0-cell $\A$ in $\Ll$ there is a group $G=G_{\A,\A}$, a monoidal category $\Ll(\A,\A)=
\displaystyle{\bigcup_{\alpha\in G}^{\bullet}}\Ll(\A, \A)_{\alpha}$ by 6 $i)$ and 7 $i)$ and the functors $\Fi_\beta^{\A,\A}$ are monoidal by 7 ii). 
If the isomorphisms $s^{\A,\A,\A}_{Y,X}$ and $s_0^\A$ from 6 $ii)$ 
are identities, $\Ll(\A,\A)$ is a Turaev category. 
Moreover, if for every pair of 0-cells $\A,\B\in\Ll^o$ the group $G_{\A,\B}$ is trivial, then $\Ll$ is a bicategory. 

\medskip

In view of the latter, if in a Turaev bicategory $\Ll$ for every $\A\in\Ll^0$ it is $\A=\displaystyle{\bigcup_{\alpha\in G_{\A,\A}}^{\bullet} } \A_\alpha$ and 
for every $\A,\B\in\Ll^0$ it is $\Ll(\A, \B)_e=\Ll(\A_e, \B_e)_e$, then $\Ll$ has an underlying bicategory $\U(\Ll)$ and  
there is a forgetful pseudofunctor from $\Ll$ to $\U(\Ll)$. 
We define them as follows. 
For every 0-cell $\A=\displaystyle{\bigcup_{\alpha\in G_{\A,\A}}^{\bullet} } \A_\alpha$ in $\Ll$ let $\A_e$ be a 0-cell of $\U(\Ll)$. 
For every two 0-cells $\A,\B\in(\U(\Ll))^0$ we define the category $\U(\Ll)(\A, \B)=\Ll(\A, \B)_e$. 
The forgetful pseudofunctor $\F: \Ll\to\U(\Ll)$ is defined as the obvious one-to-one correspondence on 0-cells and for every two 0-cells $\A,\B\in\Ll^0$ the functor 
$\F_{\A,\B}: \Ll(\A, \B)=\displaystyle{\bigcup_{\alpha\in G_{\A,\B}}^{\bullet} } \Ll(\A, \B)_{\alpha} \to \Ll(\A, \B)_e$ sends the categories 
$\Ll(\A, \B)_{\alpha}$ with $\alpha\not=e$ to the empty category and it is identity on $\Ll(\A, \B)_e$. 

The other way around, if $\Ll$ is a bicategory and there is a Turaev bicategory $\Ll^T$ such that $\Ll$ is the underlying bicategory of $\Ll^T$, 
we say that  $\Ll^T$ is an {\em extension of $\Ll$ to a Turaev bicategory}, or a {\em Turaev extension of $\Ll$}. 

\bigskip

\begin{rem}
We can consider the following version of the definition of a Turaev bicategory. Let $\Cat(G)$ denote the monoidal category whose objects are the 
elements of a group $G$, the only morphisms are the identities and the tensor product is given by the product in $G$. For a category $\C$ 
let $\u{\Aut}(\C)$ denote the monoidal category of auto-equivalences of $\C$ and natural isomorphisms between them, where the tensor product is given
by the composition of functors. 

Now, in the other version of the above definition substitute the group map $\Fi^{\A,\B}: G_{\A,\B}\to\Aut(\Ll(\A, \B)), \beta\mapsto\Fi^{\A,\B}_\beta$ 
from the point 3 $ii)$ by a monoidal functor $\Fi^{\A,\B}: \Cat(G_{\A,\B})\to\u{\Aut}(\Ll(\A, \B)), \beta\mapsto\Fi^{\A,\B}_\beta$. This means that for 
every $\alpha,\beta\in G_{\A,\B}$ there is an isomorphism $r_{\alpha,\beta}: \Fi_\alpha\comp \Fi_\beta\stackrel{\iso}{\to}\Fi_{\alpha\beta}$ 
defining the monoidal structure of the functor $\Cat(G_{\A,\B})\to\u{\Aut}(\Ll(\A, \B))$ ($r_{\alpha,\beta}$ satisfies the coherence hexagon). 

As above, we have that $\Ll(\A, \A)$ is a monoidal category and that the functors $\Fi_\beta^{\A,\A}$ are monoidal. In this setting, though, this 
means that there is an action of the group $G_{\A,\A}$ on the category $\Ll(\A,\A)$ (the definition of an action of a group on a monoidal category 
goes back to \cite{Del}.) 
In particular, one may consider a crossed product category $\Ll(\A,\A)\rtimes G$, \cite[Section 3.1]{Nik}, \cite{Tamb1}. 

This definition of a Turaev bicategory is more general: to pass from it to the former consider the truncations of the categories, {\em i.e.} 
categorical groups $\Cat(G_{\A,\B})$ and $\u{\Aut}(\Ll(\A, \B))$ to get the corresponding groups and recover the group map 
$\Fi^{\A,\B}: G_{\A,\B}\to\Aut(\Ll(\A, \B)), \beta\mapsto\Fi^{\A,\B}_\beta$ from 3 $ii)$. (By truncation we mean forgetting the morphisms 
and identifying the isomorphic objects of the respective categories. Then clearly the category $\Cat(G_{\A,\B})$ yields the group $G_{\A,\B}$ and the 
category of auto-equivalences $\u{\Aut}(\Ll(\A, \B))$ comes down to the group of automorphisms $\Aut(\Ll(\A, \B))$ of the category $\Ll(\A, \B)$.) 
\end{rem}

\bigskip

Let us define pseudofunctors between Turaev bicategories.

\begin{defn}
A Turaev pseudofunctor $\F:\Ll\to\Pp$ between Turaev bicategories $(\Ll, \Fi^\Ll,a^\Ll,\lambda^\Ll,\rho^\Ll)$ and $(\Pp, \Fi^\Pp,a^\Pp,\lambda^\Pp,\rho^\Pp)$ consist of the following data: 
\begin{enumerate}
\item an assignment $\A\mapsto\F(\A)$ for every object $\A\in\Ll^o$;
\item a family of group maps $(\psi_{\A,\B}: G_{\A,\B}\to G_{\F(\A),\F(\B)} \vert \A,\B\in\Ll^o)$; 
\item for any pair of objects $\A, \B\in\Ll^o$ a functor $\F_{\A,\B}: \Ll(\A, \B)\to\Pp(\F(\A),\F(\B))$ such that for all $\alpha\in G_{\A,\B}$ it holds: 
\begin{enumerate}[i)]
\item $\F_{\A,\B}(\Ll(\A, \B)_{\alpha})\subseteq\Pp(\F(\A),\F(\B))_{\psi_{\A,\B}(\alpha)}$,
\item $\F_{\A,\B}(\Fi^\Ll_\beta(\Ll(\A, \B)_\alpha)) = \Fi^\Pp_{\psi_{\A,\B}(\beta)} (\F_{\A,\B}(\Ll(\A, \B)_\alpha))$ in 
$\Fi^\Pp_{\psi_{\A,\B}(\beta)}\left(\Pp(\F(\A), \F(\B))_{\psi_{\A,\B}(\beta\alpha\beta^{-1})}\right)$ for every $\alpha, \beta\in G_{\A,\B}$;
\end{enumerate}
\item for all $\A,\B,\C\in\Ll^o$ and $\alpha\in G_{\A,\B}, \beta\in G_{\B,\C}$ there are isomorphisms (2-cells) natural in composable 1-cells 
$Y\in\Ll(\B, \C)_\beta, X\in\Ll(\A, \B)_\alpha$:  
$$\xi^{\beta,\alpha}_{Y,X}: \F_{\B, \C}^\beta(Y)\cdot\F_{\A, \B}^\alpha(X) \to \F_{\A, \C}^{\beta\alpha}(Y\cdot X)$$
and for all $\A\in\Ll^o$ there is an isomorphism (2-cell) natural in $\A$: 
$$\xi^0_\A: \Id_{\F(\A)}\to\F_{\A,\A}^e(\Id_\A),$$
 where $e$ is the unit element of $G_{\A,\A}$ 
(the functors with supra-indexes $\F^{\bullet}_{-, -}$ are the obvious restrictions of the functors $\F_{-, -}$\s, similarly in 
$\xi^{\beta,\alpha}_{\C,\B,\A}$\s, the supra-indexes may be omitted),  
such that the following hexagonal and triangular diagrams commute for all composable 1-cells 
$Z\in\Ll(\C, \D)_\gamma, Y\in\Ll(\B, \C)_\beta, X\in\Ll(\A, \B)_\alpha$ and all $\alpha\in G_{\A,\B}, \beta\in G_{\B,\C},\gamma\in G_{\C,\D}$: 
\smallskip
\begin{equation*}
\scalebox{0.84}{
\bfig
\putmorphism(-200,500)(1,0)[(\F_{\C, \D}(Z)\cdot \F_{\B, \C}(Y))\cdot \F_{\A, \B}(X) ` \F_{\B, \D}(Z\cdot Y)\cdot\F_{\A, \B}(X) ` 
   \xi_{Z,Y}^{\gamma,\beta}\times\id]{1620}1a
\putmorphism(1420,500)(1,0)[\phantom{\F_{\B, \D}(ZY)\comp\F_{\A, \B}(X)} ` \F_{\A, \D}((Z\cdot Y)\cdot X) ` \xi_{Z\cdot Y,X}^{\gamma\beta \s,\alpha}]{1100}1a
\putmorphism(2470,500)(0,-1)[` ` \F_{\A, \D}( a^\Ll)]{500}1l
\putmorphism(-160,500)(0,-1)[``a^\Pp]{500}1l
\putmorphism(-200,0)(1,0)[\F_{\C, \D}(Z)\cdot(\F_{\B, \C}(Y)\cdot\F_{\A, \B}(X)) ` \F_{\C, \D}(Z)\cdot\F_{\A, \C}(Y\cdot X) ` 
   \id\times \xi_{Y,X}^{\beta,\alpha}]{1620}1b
\putmorphism(1420,0)(1,0)[\phantom{\F_{\C, \D}(Z)\comp\F_{\A, \C}(YX)} ` \F_{\A, \D}(Z\cdot (Y\cdot X)) ` \xi_{Z,Y\cdot X}^{\gamma,\beta\alpha}]{1100}1b
\efig}
\end{equation*}

\bigskip

$$\scalebox{0.84}{
\bfig
\putmorphism(-230,500)(1,0)[\Id_{\F(\B)}\cdot\F_{\A,\B}(X) ` \F_{\B,\B}(\Id_\B)\cdot\F_{\A,\B}(X) ` \xi^0_{\B}\times\id]{1230}1a
\putmorphism(1000,500)(1,0)[\phantom{ \F_{\B,\B}(\Id_\B)\cdot\F_{\A,\B}(X)} ` \F_{\A,\B}(\Id_\B\cdot X) ` \xi_{\Id_\B,X}^{e,\alpha}]{1080}1a
\putmorphism(20,500)(2,-1)[`\F_{\A,\B}(X)`]{930}1l
\putmorphism(-250,500)(2,-1)[``\lambda^\Pp]{930}0l
\putmorphism(2060,500)(-2,-1)[``\F_{\A,\B}(\lambda^\Ll)]{1030}1r
\efig}
$$

\smallskip

$$\scalebox{0.84}{
\bfig
\putmorphism(-230,500)(1,0)[\F_{\A,\B}(X)\cdot\Id_{\F(\A)} ` \F_{\A,\B}(X)\cdot\F_{\A,\A}(\Id_\A) ` \id\times\xi^0_{\A}]{1230}1a
\putmorphism(1000,500)(1,0)[\phantom{ \F_{\B,\B}(\Id_\B)\comp\F_{\A,\B}(X)} ` \F_{\A,\B}(X\cdot\Id_\A) ` \xi_{X,\Id_\A}^{\alpha,e}]{1080}1a
\putmorphism(20,500)(2,-1)[`\F_{\A,\B}(X).`]{930}1l
\putmorphism(-250,500)(2,-1)[``\rho^\Pp]{930}0l
\putmorphism(2060,500)(-2,-1)[``\F_{\A,\B}(\rho^\Ll)]{1030}1r
\efig}
$$
\end{enumerate}
\end{defn}

If for $\A\in\Ll^0$ the groups $G_{\A,\A}$ and $G_{\F(\A),\F(\A)}$ are equal, then $\F_{\A,\A}$ is a Turaev functor between Turaev categories 
$\Ll(\A,\A)$ and $\Pp(\F(\A),\F(\A))$. On the other hand, if for all $\A,\B\in\Ll^0$ and $\A',\B'\in\Pp^0$ the groups $G_{\A,\B}$ and $G_{\A',\B'}$ are trivial, 
then a Turaev pseudofunctor is a pseudofunctor between bicategories $\Ll(\A,\B)$ and $\Pp(\A',\B')$.

\section{Turaev 2-category for bimonads in $\K$}

Let $\K$ be a 2-category. We recalled our definition of bimonads in $\K$ and their 2-category $\Bimnd(\K)$ in the preliminary Section. 
In this Section we are going to construct a Turaev extension of $\Bimnd(\K)$. In this Turaev 2-category the 0- and 1-cells are 0- and 1-cells of $\Bimnd(\K)$ 
expanded and twisted by certain automorphisms of bimonads. We start by presenting the automorphisms that we are going to be interested in.

\subsection{A subgroup of the group of automorphisms of a bimonad in a 2-category}

Let $(\A, F, \lambda)\cong(\A, F)$ be a bimonad in $\K$ and consider an invertible 1-cell $\alpha: F\to F$ in $\K$. We observe that 
$(\Id_\A, \alpha): (\A, F)\to(\A, F)$ is a 1-cell in $\Mnd(\K)$ if and only if 
\begin{equation} \eqlabel{alg m}
\scalebox{0.86}{
\gbeg{3}{4}
\got{1}{F} \got{3}{F} \gnl
\gwmu{3} \gnl
\gvac{1} \gbmp{\alpha} \gnl
\gvac{1} \gob{1}{F}
\gend }=
\scalebox{0.86}{
\gbeg{3}{4}
\got{1}{F} \got{3}{F} \gnl
\gbmp{\alpha} \gvac{1} \gbmp{\alpha} \gnl
\gwmu{3} \gnl
\gob{3}{F}
\gend}
\qquad \textnormal{and} \qquad
\scalebox{0.86}{
\gbeg{1}{4}
\gu{1} \gnl
\gbmp{\alpha} \gnl
\gcl{1} \gnl
\gob{1}{F}
\gend} = 
\scalebox{0.86}{
\gbeg{1}{4}
\gu{1} \gnl
\gcl{2} \gnl
\gob{1}{F}
\gend}
\end{equation} 
holds. Similarly, $(\Id_\A, \alpha): (\A, F)\to(\A, F)$ is a 1-cell in $\Comnd(\K)$ if and only if 
\begin{equation} \eqlabel{coalg m}
\scalebox{0.86}{
\gbeg{3}{4}
\got{3}{F} \gnl
\gvac{1} \gbmp{\alpha} \gnl
\gwcm{3} \gnl
\gob{1}{F} \gob{3}{F} \gnl
\gend }=
\scalebox{0.86}{
\gbeg{3}{4}
\got{3}{F} \gnl
\gwcm{3} \gnl
\gbmp{\alpha} \gvac{1} \gbmp{\alpha} \gnl
\gob{1}{F} \gob{3}{F} \gnl
\gend}
\qquad \textnormal{and} \qquad
\scalebox{0.86}{
\gbeg{1}{4}
\got{1}{F} \gnl
\gcl{1} \gnl
\gbmp{\alpha} \gnl
\gcu{1} \gnl
\gend} = 
\scalebox{0.86}{
\gbeg{1}{4}
\got{1}{F} \gnl
\gcl{2} \gnl\gnl
\gcu{1} \gnl
\gend}
\end{equation}
holds. Composing \equref{alg m} with $\alpha^{-1}\times\alpha^{-1}$ from above and with $\alpha^{-1}$ from below (and dually for \equref{coalg m}) 
we see that $(\Id_\A, \alpha): (\A, F)\to(\A, F)$ is a 1-cell in $\Mnd(\K)$ if and only if so is $(\Id_\A, \alpha^{-1}): (\A, F)\to(\A, F)$, and that 
$(\Id_\A, \alpha): (\A, F)\to(\A, F)$ is a 1-cell in $\Comnd(\K)$ if and only if so is $(\Id_\A, \alpha^{-1}): (\A, F)\to(\A, F)$. Then it makes sense to consider a 1-cell 
$(\Id_\A, \alpha, \alpha^{-1})$ in $\Bimnd(\K)$. The additional condition that the latter fulfills is:
\begin{equation} \eqlabel{YD cond for alfa}
\gbeg{3}{5}
\got{1}{F} \got{1}{F} \gnl
\gbmp{\alpha} \gcl{1} \gnl
\glmptb \gnot{\hspace{-0,34cm}\lambda} \grmptb \gnl
\gbmp{\s\alpha^{\w-1}} \gcl{1} \gnl
\gob{1}{F} \gob{1}{F}
\gend=
\gbeg{3}{5}
\got{1}{F} \got{1}{F} \gnl
\gcl{1} \gbmp{\s\alpha^{\w-1}} \gnl
\glmptb \gnot{\hspace{-0,34cm}\lambda} \grmptb \gnl
\gcl{1} \gbmp{\alpha} \gnl
\gob{1}{F} \gob{1}{F.}
\gend
\end{equation}
Then similarly as above we have that $(\Id_\A, \alpha, \alpha^{-1})$ is a 1-cell in $\Bimnd(\K)$ if and only if so is $(\Id_\A, \alpha^{-1}, \alpha)$, 
and these two 1-cells are inverse to each other. Now we define $G(F)=\Aut_0(F)$ to be the group of 1-endocells in $\Bimnd(\K)$ on the 0-cell $(\A, F)$ 
of the form $(\Id_\A, \alpha, \alpha^{-1})$, where $\alpha:F\to F$ is an invertible 1-cell in $\K$. It is a subgroup of the group of automorphisms of 
the 0-cell $(\A, F)$ in $\Bimnd(\K)$. We will refer to such automorphisms as to {\em 0-automorphisms of $F$}.

\subsection{Fusion group maps}

Let us consider the partition of the class $\Bimnd(\K)^0$ of all bimonads in $\K$ into classes of bimonads whose 0-automorphism groups are isomorphic. That is:
$$\Bimnd(\K)^0=\displaystyle{\bigcup_{\omega\in \I}^{\bullet}\Bimnd_\omega(\K)^0} $$
for some index class $\I$, so that for every $F,F\s'\in \Bimnd_\omega(\K)^0$ it is $G(F)\iso G(F\s')$. In every class $\Bimnd_\omega(\K)^0$ we will choose a 
family of group maps $(j_{F,F\s'}: G(F)\to G(F\s') \vert \forall F, F\s'\in \Bimnd(\K)^0)$ such that 
\begin{enumerate}
\item $j_{F,F}=id_{G(F)}$, 
\item $j_{F,F\s'}=(j_{F\s',F})^{-1}$, 
\item $j_{F,F\s''}=j_{F\s',F\s''}\comp j_{F,F\s'}$
\end{enumerate}
for all $F, F\s', F\s''\in \Bimnd_\omega(\K)^0)$. In other words, in every class $\Bimnd_\omega(\K)^0$ we have a directed system of groups with 
group isomorphisms. The group isomorphisms $j$ we will call {\em fusion maps}. 

\medskip

Given two bimonads $F, F\s'$ in $\K$ with their respective groups $G, G'$ and a fusion map $j: G\to G'$ it is directly checked that 
\begin{equation} \eqlabel{product G'G}
(\alpha,\beta)*(\gamma,\delta)=(\alpha\gamma,\delta j^{-1}(\gamma^{-1})\beta j^{-1}(\gamma))
\end{equation} 
with $\alpha, \gamma\in G', \beta, \delta\in G$, defines an associative multiplication on the product of groups $G'\times G$. 
The unit for this multiplication is given by $(e_{G'},e_G)$, where $e_{G'},e_G$ denote the unit of the corresponding groups, 
and it is 
$$(\alpha,\beta)^{-1}=(\alpha^{-1}, j^{-1}(\alpha)\beta^{-1}j^{-1}(\alpha^{-1})).$$

\medskip

Given three bimonads $F, F\s', F\s''$ in $\K$ with their respective groups $G, G', G'''$ and fusion maps $j: G\to G', j':G'\to G''$ 
we can consider groups $G''\times G', G''\times G$ with the analogous product as in \equref{product G'G} (where $j$ is substituted by $j\s'$ in the former case, 
and by $j\s^{02}=j\s' j$ in the latter).

We will consider the transitive product on the groups
$$ m^{G'', G',G}_{j\s',j}: (G''\times G') \times (G'\times G) \to (G''\times G)$$
defined by 
\begin{equation} \eqlabel{trans product}
(\alpha,\beta)*(\gamma,\delta)=(\alpha j\s'(\gamma),\delta\cdot j^{-1}(\gamma^{-1}\beta\gamma))
\end{equation} 
for $\alpha\in G'', \beta, \gamma\in G', \delta\in G$. We will use the same symbol $*$ for the product in $G' \times G$ from \equref{product G'G} and the above defined transitive product,
the difference will be clear from the context. Similarly as above, it is directly proved that if we are given a fourth bimonad $F\s'''$ in $\K$ with 
its corresponding group $G'''$ and a fusion map $j\s'': G''\to G'''$, then the transitive product is associative. This means that setting $j\s^{13}=j\s''j\s'$, we have that 
the composition of maps: 
$$\left((G'''\times G'') \times(G''\times G')\right) \times (G'\times G) \to (G''\times G) \stackrel{m^{G''', G'', G'}_{j\s'',j\s'}\times\id}{\longrightarrow} (G'''\times G')\times (G'\times G) 
\stackrel{m^{G''', G', G}_{j\s^{13},j}}{\longrightarrow} (G'''\times G)$$
is equal to the composition:
$$(G'''\times G'') \times \left((G''\times G') \times (G'\times G)\right) \to (G''\times G) \stackrel{\id\times m^{G'', G', G}_{j\s',j}}{\longrightarrow} (G'''\times G'')\times (G''\times G) 
\stackrel{m^{G''', G'', G}_{j\s'',j\s^{02}}}{\longrightarrow} (G'''\times G).$$ 
Concretely, given $(\alpha,\beta)\in G'''\times G'', (\gamma,\delta)\in G''\times G', (\mu,\nu)\in G'\times G$ we find that: 
\begin{equation} \eqlabel{triple trans prod}
\left((\alpha, \beta)*(\gamma, \delta)\right)*(\mu,\nu)=
(\alpha j\s''(\gamma) j\s^{13}(\mu), \nu j^{-1}\left(\mu^{-1}\delta j\s'^{-1}(\gamma^{-1}\beta\gamma)\mu\right))
\end{equation}
and that 
$$(\alpha, \beta)*\left((\gamma, \delta)*(\mu,\nu)\right)=(\alpha j\s''(\gamma j\s'(\mu)), \nu j^{-1}(\mu^{-1}\delta\mu) 
(j\s^{02})^{-1}\left(j\s'(\mu^{-1}) \gamma^{-1}\beta\gamma j\s'(\mu)\right)),$$
which are the same. Moreover, we have $(e_{G''}, e_{G'})*(\gamma, \delta)=(j\s'(\gamma),\delta)$ and $(\alpha, \beta)*(e_{G'}, e_G)=(\alpha, j^{-1}(\beta))$.

Furthermore, we will consider the projection maps 
$$\pi_1: G''\times G \to G''\times G' \qquad\text{and}\qquad \pi_2: G''\times G \to G'\times G$$
given by 
$$\pi_1=\id\times j \qquad\text{and}\qquad \pi_2=(j\s')^{-1}\times\id.$$
They are group maps and they induce the group map 
\begin{equation} \eqlabel{pi for bim}
\pi: G''\times G \to (G''\times G')\times (G'\times G), (\alpha, \beta)\mapsto \left((\alpha, j(\beta)), ((j\s')^{-1}(\alpha), \beta)\right).
\end{equation} 
It is directly checked that $\pi$ satisfies the identities \equref{conjug compat}--\equref{counit law}. 

\smallskip

Thus we have constructed a transitive system of groups with projections in every class $\Bimnd_\omega(\K)^0$. 
If $F$ and $F\s'$ are bimonads in $\K$ so that $G(F)=G$ and $G(F\s')=G'$ are not isomorphic, we set $G_{F,F\s'}=\{e\}$. 
The total transitive system of groups with projections over $\Bimnd(\K)$ we may denote by $(\{G_{F,F\s'}\s\vert\s F,F\s'\in\Bimnd(\K)\}, *, \pi)$. 
Here $G_{F,F\s'}=G'\times G$.

\begin{rem}
When $\K$ is the 2-category induced by the (braided) monoidal category of vector spaces over a field $k$ bimonads in $\K$ are $k$-bialgebras. The 0-automorphism 
groups of bimonads in $\K$ are bialgebra automorphism groups. Just to mention some, in \cite{Shil} is presented a list of references where some automorphism groups 
of Hopf algebras are computed, few are computed also in \cite{VZ4}. For example for the Sweedler's four-dimensional Hopf algebra the automorphism group is isomorphic to 
$k^*=k\backslash\{0\}$, and for the Radford's Hopf algebra of dimension $m2^{n+1}$ over the complex numbers field $C$  the automorphism group is isomorphic to $GL_n(C)$. 
\end{rem}

\subsection{Turaev 2-category for bimonads in $\K$}

We define the 2-category $\Bimnd^T(\K)$  in the following way. 

\underline{0-cells:} are triples $(\A, F, (\lambda_{\alpha})_{\alpha\in G})=\displaystyle{\bigcup_{\alpha\in G}^{\bullet} (\A, F, \lambda_{\alpha})}$ where: 
\begin{itemize}
\item $G=\Aut_0(F)$, 
\item $(\A, F, \lambda_e)$ is a bimonad in $\K$, 
\item for every $\alpha\in G$ the pair $(F, \lambda_{\alpha})$ is a 1-cell both in $\Mnd(K)$ and in $\Comnd(K)$ so that 
\begin{equation} \eqlabel{lambda alfa-beta}
\lambda_{\beta\alpha}=
\gbeg{3}{5}
\got{1}{F} \got{1}{F} \gnl
\gbmp{\beta} \gbmp{\s\alpha^{\w-1}} \gnl
\glmptb \gnot{\hspace{-0,34cm}\lambda} \grmptb \gnl
\gbmp{\alpha} \gbmp{\s\beta^{\w-1}} \gnl
\gob{1}{F} \gob{1}{F}
\gend
\end{equation}
\end{itemize}

We will use the following notation: $F\equiv(\A, F, (\lambda_{\alpha})_{\alpha\in G})$ and $F_\alpha\equiv(\A, F, \lambda_{\alpha}).$

\medskip

\underline{The transitive system of groups with projections:} $(\{G'\times G\s\vert\s F,F\s'\in\Bimnd(\K)\}, *, \pi)$ from above.  

\medskip

\underline{Hom-category.} Given two 0-cells $(\A, F, (\lambda_{\alpha})_{\alpha\in G})$ and $(\A', F\s', (\lambda_{\alpha}')_{\alpha\in G})$ 
we define the category $\Bimnd^T(F,F\s'):=\Bimnd^T(\K)(F,F\s')$ as follows:
\begin{equation} \eqlabel{hom cat bim Tur}
\Bimnd^T(F,F\s')=
 \left \{
      \begin{matrix} 
         \displaystyle{\bigcup_{(\alpha,\beta)\in G'\times G}^{\bullet} \Bimnd^T(F_\beta,F\s'_\alpha)},&  \quad\text{if}\hspace{0,1cm} G\iso G' \\
         \hspace{4,2cm}\emptyset, & \text{if not }
      \end{matrix}
    \right .
\end{equation} 
so, if the 0-automorphism groups $G$ and $G'$ are not isomorphic, we set $\Bimnd^T(F,F\s')$ to be the empty category (with no objects and no morphisms), 
if $ G\iso G'$, then let $\Bimnd^T(F,F\s')$ be the disjoint union of the categories $\Bimnd^T(F_\beta,F\s'_\alpha)$ that we next define. The objects of 
$\Bimnd^T(F_\beta,F\s'_\alpha)$ are triples $(X, \psi_X, \phi_X)$ so that: 
\begin{itemize}
\item $(X, \psi_X: F\s'X\to XF\s')$ is a 1-cell in $\Mnd(\K)$, 
\item $(X, \phi_X:XF\to F\s'X)$ is a 1-cell in $\Comnd(\K)$, and 
\item the compatibility condition 
\begin{equation}  \eqlabel{twisted YD cond}
\gbeg{3}{5}
\got{1}{F\s'} \got{1}{X} \got{1}{F} \gnl
\glmptb \gnot{\hspace{-0,34cm}\psi_X} \grmptb \gcl{1} \gnl
\gcl{1} \glmptb \gnot{\hspace{-0,34cm}\lambda_\beta} \grmptb \gnl
\glmptb \gnot{\hspace{-0,34cm}\phi_X} \grmptb \gcl{1} \gnl
\gob{1}{F\s'} \gob{1}{X} \gob{1}{F}
\gend=
\gbeg{3}{5}
\got{1}{F\s'} \got{1}{X} \got{1}{F} \gnl
\gcl{1} \glmptb \gnot{\hspace{-0,34cm}\phi_X} \grmptb \gnl
\glmptb \gnot{\hspace{-0,34cm}\lambda'_\alpha} \grmptb \gcl{1} \gnl
\gcl{1} \glmptb \gnot{\hspace{-0,34cm}\psi_X} \grmptb \gnl
\gob{1}{F\s'} \gob{1}{X} \gob{1}{F}
\gend
\end{equation} 
holds.
\end{itemize}

Morphisms of $\Bimnd^T(F_\beta,F\s'_\alpha)$ are $\zeta: (X, \psi_X, \phi_X)\to (Y, \psi_Y, \phi_Y)$ where 
$\zeta: (X, \psi_X)\to (Y, \psi_Y)$ is a 2-cell in $\Mnd(\K)$ and $\zeta: (X, \phi_X)\to (Y, \phi_Y)$ is a 2-cell in $\Comnd(\K)$. 

\medskip

For every 0-cell $(\A, F, (\lambda_{\alpha})_{\alpha\in G})$ of $\Bimnd^T(\K)$ the identity 1-cell on it is given by $(\Id_\A, \id_F, \id_F)$ 
living in the category $\Bimnd^T(F_e,F_e)$.  

\medskip

For every 1-cell $(X, \psi_X, \phi_X)$ in $\Bimnd^T(\K)$ the identity 2-cell on it is given by $\id_X$. 

\medskip

For the composition of 1-cells we will first have to prove a few results. We start by recording some direct consequences of the identity \equref{lambda alfa-beta}:
\begin{center} 
\begin{tabular}{p{3cm}p{0cm}p{3cm}p{0cm}p{3.4cm}p{0cm}p{3.4cm}}
\begin{equation} \eqlabel{l1}
\lambda_\alpha=
\gbeg{2}{5}
\got{1}{F} \got{1}{F} \gnl
\gcl{1} \gbmp{\s\alpha^{\w-1}} \gnl
\glmptb \gnot{\hspace{-0,34cm}\lambda} \grmptb \gnl
\gbmp{\alpha} \gcl{1} \gnl
\gob{1}{F} \gob{1}{F}
\gend
\end{equation} & & 
\begin{equation} \eqlabel{l2}
\lambda_\alpha=
\gbeg{2}{5}
\got{1}{F} \got{1}{F} \gnl
\gbmp{\alpha} \gcl{1} \gnl
\glmptb \gnot{\hspace{-0,34cm}\lambda} \grmptb \gnl
\gcl{1} \gbmp{\s\alpha^{\w-1}} \gnl
\gob{1}{F} \gob{1}{F}
\gend
\end{equation} & & 
\begin{equation} \eqlabel{l3}
\lambda_{\alpha\beta}=
\gbeg{2}{5}
\got{1}{F} \got{1}{F} \gnl
\gcl{1} \gbmp{\s\alpha^{\w-1}} \gnl
\glmptb \gnot{\hspace{-0,34cm}\lambda_\beta} \grmptb \gnl
\gbmp{\alpha} \gcl{1} \gnl
\gob{1}{F} \gob{1}{F}
\gend
\end{equation}  & &
\begin{equation} \eqlabel{l4}
\lambda_{\alpha\beta}=
\gbeg{2}{5}
\got{1}{F} \got{1}{F} \gnl
\gbmp{\beta} \gcl{1} \gnl
\glmptb \gnot{\hspace{-0,34cm}\lambda_\alpha} \grmptb \gnl
\gcl{1} \gbmp{\s\beta^{\w-1}} \gnl
\gob{1}{F} \gob{1}{F}
\gend
\end{equation}
\end{tabular}
\end{center}

\begin{center} 
\begin{tabular}{p{5cm}p{0cm}p{5cm}}  
\begin{equation} \eqlabel{l3'}
\gbeg{2}{4}
\got{1}{F} \got{1}{F} \gnl
\glmptb \gnot{\hspace{-0,34cm}\lambda_{\alpha\beta}} \grmptb \gnl
\gcl{1} \gbmp{\beta} \gnl
\gob{1}{F} \gob{1}{F}
\gend=
\gbeg{2}{4}
\got{1}{F} \got{1}{F} \gnl
\gbmp{\beta} \gcl{1} \gnl 
\glmptb \gnot{\hspace{-0,34cm}\lambda_\alpha} \grmptb \gnl
\gob{1}{F} \gob{1}{F}
\gend
\end{equation} & & 
\begin{equation*} 
\gbeg{2}{4}
\got{1}{F} \got{1}{F} \gnl
\gcl{1} \gbmp{\alpha} \gnl
\glmptb \gnot{\hspace{-0,34cm}\lambda_{\alpha\beta}} \grmptb \gnl
\gob{1}{F} \gob{1}{F}
\gend=
\gbeg{2}{4}
\got{1}{F} \got{1}{F} \gnl
\glmptb \gnot{\hspace{-0,34cm}\lambda_\beta} \grmptb \gnl
\gbmp{\alpha} \gcl{1} \gnl 
\gob{1}{F} \gob{1}{F}
\gend
\end{equation*}
\end{tabular}
\end{center}

\begin{center} 
\begin{tabular}{p{3.4cm}p{0cm}p{3.4cm}p{0cm}p{3.4cm}} 
\begin{equation} \eqlabel{l7}
\gbeg{2}{4}
\got{1}{F} \got{1}{F} \gnl
\gcl{1} \gbmp{\s\alpha^{\w-1}} \gnl
\glmptb \gnot{\hspace{-0,34cm}\lambda} \grmptb \gnl
\gob{1}{F} \gob{1}{F}
\gend=
\gbeg{2}{4}
\got{1}{F} \got{1}{F} \gnl
\glmptb \gnot{\hspace{-0,34cm}\lambda_\alpha} \grmptb \gnl
\gbmp{\s\alpha^{\w-1}} \gcl{1} \gnl 
\gob{1}{F} \gob{1}{F}
\gend
\end{equation} & &
\begin{equation} \eqlabel{l5}
\gbeg{2}{4}
\got{1}{F} \got{1}{F} \gnl
\gbmp{\alpha} \gcl{1} \gnl
\glmptb \gnot{\hspace{-0,34cm}\lambda} \grmptb \gnl
\gob{1}{F} \gob{1}{F}
\gend=
\gbeg{2}{4}
\got{1}{F} \got{1}{F} \gnl
\glmptb \gnot{\hspace{-0,34cm}\lambda_\alpha} \grmptb \gnl
\gcl{1} \gbmp{\alpha} \gnl
\gob{1}{F} \gob{1}{F}
\gend
\end{equation} & & 
\begin{equation} \eqlabel{l6}
\gbeg{2}{4}
\got{1}{F} \got{1}{F} \gnl
\gbmp{\s\alpha^{\w-1}} \gcl{1} \gnl
\glmptb \gnot{\hspace{-0,34cm}\lambda_\alpha} \grmptb \gnl
\gob{1}{F} \gob{1}{F}
\gend=
\gbeg{2}{4}
\got{1}{F} \got{1}{F} \gnl
\glmptb \gnot{\hspace{-0,34cm}\lambda} \grmptb \gnl
\gcl{1} \gbmp{\s\alpha^{\w-1}} \gnl
\gob{1}{F} \gob{1}{F}
\gend
\end{equation} 
\end{tabular}
\end{center}

\begin{equation} \eqlabel{l8}
\gbeg{3}{4}
\gvac{1} \got{1}{F} \got{1}{F} \gnl
\glmp \gnot{\hspace{-0,34cm}\alpha^{-1}\beta} \grmptb \gcl{1} \gnl
\gvac{1} \glmptb \gnot{\hspace{-0,34cm}\lambda_\alpha} \grmptb \gnl
\gvac{1} \gob{1}{F} \gob{1}{F}
\gend=
\gbeg{2}{4}
\got{1}{F} \got{1}{F} \gnl
\glmptb \gnot{\hspace{-0,34cm}\lambda_\beta} \grmptb \gnl
\gcl{1} \glmptb \gnot{\hspace{-0,34cm}\alpha^{-1}\beta} \grmp \gnl
\gob{1}{F} \gob{1}{F}
\gend
\end{equation}

\begin{lma} \lelabel{twisted psi's}
Given a 1-cell $(X, \psi_X, \phi_X): (\A, F, (\lambda_{\alpha})_{\alpha\in G}) \to (\A', F\s', (\lambda_{\alpha}')_{\alpha\in G'})$ in $\Bimnd^T(\K)$, 
where $G\iso G'$, and let $j: G\to G'$ be the corresponding fusion map. 
\begin{enumerate}[a)]
\item For every $\alpha\in G$ the pair $(X, \psi_X^\alpha)$ is a 1-cell in $\Mnd(\K)$, where $\psi_X^\alpha$ is given by \equref{G-twisted}.
\item For every $\beta\in G'$ the pair $(X, \psi_X^\beta)$ is a 1-cell in $\Mnd(\K)$, where $\psi_X^\beta$ is given by \equref{G'-twisted}. 
\begin{center} 
\begin{tabular}{p{5cm}p{2cm}p{5cm}}
\begin{equation} \eqlabel{G-twisted}
\psi_X^\alpha=
\gbeg{3}{5}
\gvac{1} \got{1}{F\s'} \got{1}{X} \gnl
\glmp \gnot{\hspace{-0,34cm}j\s(\alpha)} \grmptb \gcl{1} \gnl
\gvac{1} \glmptb \gnot{\hspace{-0,34cm}\psi_X} \grmptb \gnl
\gvac{1} \gcl{1} \gbmp{\s\alpha^{\w-1}}  \gnl
\gvac{1} \gob{1}{X} \gob{1}{F}
\gend
\end{equation} & & 
\begin{equation} \eqlabel{G'-twisted}
\psi_X^\beta=
\gbeg{3}{5}
\got{1}{F\s'} \got{1}{X} \gnl
\gbmp{\beta} \gcl{1} \gnl
\glmptb \gnot{\hspace{-0,34cm}\psi_X} \grmptb \gnl
\gcl{1} \glmptb \gnot{j^{-1}(\beta^{-1})} \gcmp \grmp \gnl
\gob{1}{X} \gob{1}{F}
\gend
\end{equation}
\end{tabular}
\end{center}
\end{enumerate}
\end{lma}

\begin{proof}
The proof is direct using the fact that $j(\alpha), \alpha^{\w-1}, \beta$ and $j^{-1}(\beta^{-1})$ fulfill the identities \equref{alg m}. 
\qed\end{proof}

\begin{rem}
Observe that the above result holds for a 1-cell $(X, \psi_X): (\A, B)\to(\A', B')$ in $\Mnd(\K)$, where $B$ and $B'$ are monads in $\K$ with the corresponding 0-zero automorphism groups 
$G$ and $G'$, which are isomorphic and consist of 1-cells $(\Id_\A, \alpha): (\A, B)\to(\A', B')$ in $\Mnd(\K)$, and $j: G\to G'$ is any group isomorphism. 
\end{rem}

The composition of 1-cells in $\Bimnd^T(\K)$ is defined as it is indicated in the following Proposition.

\begin{prop} \prlabel{comp 1-cells}
Let $(X, \psi_X, \phi_X) \in \Bimnd^T(F\s'_\beta, F\s''_\alpha)$ and $(Y, \psi_Y, \phi_Y)\in \Bimnd^T(F_\delta,F\s'_\gamma)$ be composable 1-cells in $\Bimnd^T(\K)$, 
and let $j:G\to G'$ and $j\s': G'\to G''$ be the corresponding fusion maps. 
The triple $(XY, \psi_{XY}, \phi_{XY})$, where $\psi_{XY}$ and $\phi_{XY}$ 
are given as below, is a 1-cell in $\Bimnd^T(\K)$, in particular it is an object in 
$\Bimnd^T(F_{\delta\cdot j^{-1}(\gamma^{-1}\beta\gamma)},F\s''_{\alpha\cdot j\s'(\gamma)})$.
\begin{center} 
\begin{tabular}{p{6cm}p{0cm}p{5cm}} 
\begin{equation} \eqlabel{psi_YX}
\psi_{XY}=
\gbeg{3}{4}
\got{1}{F\s''} \got{1}{X} \got{1}{Y} \gnl
\glmptb \gnot{\hspace{-0,34cm}\psi_X^\gamma} \grmptb \gcl{1} \gnl
\gcl{1} \glmptb \gnot{\psi_Y^{ \gamma^{-1}\beta\gamma}} \gcmptb \grmp \gnl 
\gob{1}{X} \gob{1}{Y} \gob{1}{F}
\gend
\end{equation} & & 
\begin{equation} \eqlabel{phi_YX}
\phi_{XY}=
\gbeg{3}{4}
\got{1}{X} \got{1}{Y} \got{1}{F} \gnl
\gcl{1} \glmptb \gnot{\hspace{-0,34cm}\phi_Y} \grmptb \gnl
\glmptb \gnot{\hspace{-0,34cm}\phi_X} \grmptb \gcl{1} \gnl
\gob{1}{F\s''} \gob{1}{X} \gob{1}{Y} 
\gend
\end{equation} 
\end{tabular}
\end{center}
\end{prop}

\begin{proof}
Note that by \leref{twisted psi's} the pair $(XY, \psi_{XY})$ is a composition of 1-cells in $\Mnd(\K)$, so it is a 1-cell in $\Mnd(\K)$. We clearly have that 
$(XY, \phi_{XY})$ is a 1-cell in $\Comnd(\K)$. We only should check that the identity \equref{twisted YD cond} 
is fulfilled. We first observe that $\alpha\in G'',\beta,\gamma\in G',\delta\in G$, hence $\gamma, \gamma^{-1}\beta\gamma\in G'$. Then 
$\psi_X^\gamma$ is defined as in \equref{G-twisted}, while $\psi_Y^{\gamma^{-1}\beta\gamma}$ is defined as in \equref{G'-twisted}. Henceforth we have: 
\begin{equation} \eqlabel{psi_XY real}
\psi_{XY}=
\gbeg{9}{8}
\gvac{1} \got{1}{F\s''} \got{1}{X} \got{5}{Y} \gnl
\glmp \gnot{\hspace{-0,34cm}j\s'(\gamma)} \grmptb \gcl{1} \gvac{2} \gcl{3} \gnl
\gvac{1} \glmptb \gnot{\hspace{-0,34cm}\psi_X} \grmptb \gnl
\gvac{1} \gcl{4} \glmpt \gnot{\hspace{-0,34cm} \cancel{ \gamma^{\w-1} } } \grmpb \gnl
\gvac{2} \glmp \gnot{\gamma^{-1}\beta \cancel{ \gamma} } \gcmpt \grmpb \gcl{1} \gnl
\gvac{4} \glmptb \gnot{\hspace{-0,34cm}\psi_Y} \grmptb \gnl
\gvac{4} \gcl{1} \glmpt \gnot{ \hspace{0,4cm} j^{-1}(\gamma^{-1}\beta^{-1}\gamma)} \gcmpb \gcmp \grmp \gnl
\gob{3}{X} \gob{3}{Y} \gob{1}{F}
\gend=
\gbeg{9}{7}
\gvac{1} \got{1}{F\s''} \got{1}{X} \got{3}{Y} \gnl
\glmp \gnot{\hspace{-0,34cm}j\s'(\gamma)} \grmptb \gcl{1} \gvac{1} \gcl{3} \gnl
\gvac{1} \glmptb \gnot{\hspace{-0,34cm}\psi_X} \grmptb \gnl
\gvac{1} \gcl{3} \glmpt \gnot{\hspace{-0,34cm}\gamma^{\w-1}\beta} \grmpb \gnl
\gvac{3} \glmptb \gnot{\hspace{-0,34cm}\psi_Y} \grmptb \gnl
\gvac{3} \gcl{1} \glmpt \gnot{ \hspace{0,4cm} j^{-1}(\gamma^{-1}\beta^{-1}\gamma)} \gcmpb \gcmp \grmp \gnl
\gob{3}{X} \gob{1}{Y} \gob{3}{F.}
\gend
\end{equation}
Then the left hand-side of \equref{twisted YD cond} in this case and applying \equref{l4} becomes: 
$$
\gbeg{5}{5}
\got{1}{F\s''} \got{1}{XY} \got{1}{F} \gnl
\glmptb \gnot{\hspace{-0,34cm}\psi_{XY}} \grmptb \gcl{1} \gnl
\gcl{1} \glmptb \gnot{\hspace{0,42cm}\lambda_{\delta j^{-1}(\gamma^{-1}\beta\gamma)}} \gcmptb \gcmp \grmp \gnl
\glmptb \gnot{\hspace{-0,34cm}\phi_{XY}} \grmptb \gcl{1} \gnl
\gob{1}{F\s''} \gob{1}{XY} \gob{1}{F}
\gend=
\gbeg{12}{12}
\gvac{1} \got{1}{F\s''} \got{1}{X} \got{3}{Y} \got{7}{F} \gnl
\glmp \gnot{\hspace{-0,34cm}j\s'(\gamma)} \grmptb \gcl{1} \gvac{1} \gcl{3} \gvac{4} \gcl{6} \gnl
\gvac{1} \glmptb \gnot{\hspace{-0,34cm}\psi_X} \grmptb \gnl
\gvac{1} \gcl{3} \glmpt \gnot{\hspace{-0,34cm}\gamma^{\w-1}\beta} \grmpb \gnl
\gvac{3} \glmptb \gnot{\hspace{-0,34cm}\psi_Y} \grmptb \gnl
\gvac{3} \gcl{1} \glmpt \gnot{ \hspace{0,4cm} \cancel{ j^{-1}(\gamma^{-1}\beta^{-1}\gamma)} } \gcmpb \gcmp \grmp \gnl
\gvac{1} \gcn{2}{4}{1}{11} \gcn{2}{3}{1}{9} \glmpt \gnot{ \hspace{0,4cm} \cancel{ j^{-1}(\gamma^{-1}\beta\gamma)} } \gcmp \gcmp \grmpb \gnl
\gvac{8} \glmptb \gnot{ \hspace{-0,34cm} \lambda_\delta} \grmptb \gnl
\gvac{8} \gcl{1} \glmptb \gnot{ \hspace{0,4cm} j^{-1}(\gamma^{-1}\beta^{-1}\gamma)} \gcmp \gcmp \grmp \gnl
\gvac{7} \glmptb \gnot{\hspace{-0,34cm}\phi_Y} \grmptb \gcl{2} \gnl
\gvac{6} \glmptb \gnot{\hspace{-0,34cm}\phi_X} \grmptb \gcl{1} \gnl
\gvac{6} \gob{1}{F\s''} \gob{1}{X} \gob{1}{Y} \gob{1}{F}
\gend\stackrel{Y}{=}
\gbeg{10}{9}
\gvac{1} \got{1}{F\s''} \got{1}{X} \gvac{1} \got{1}{Y} \got{1}{F} \gnl
\glmp \gnot{\hspace{-0,34cm}j\s'(\gamma)} \grmptb \gcl{1} \gvac{1} \gcl{3} \gcl{3} \gnl
\gvac{1} \glmptb \gnot{\hspace{-0,34cm}\psi_X} \grmptb \gnl
\gvac{1} \gcl{2} \glmpt \gnot{\hspace{-0,34cm}\gamma^{\w-1}\beta} \grmpb \gnl
\gvac{3} \gcl{1} \glmptb \gnot{\hspace{-0,34cm}\phi_Y} \grmptb \gnl
\gvac{1} \gcn{2}{2}{1}{3} \glmptb \gnot{\hspace{-0,34cm}\lambda'_\gamma} \grmptb \gcl{1} \gnl
\gvac{2} \gvac{1} \gcl{1} \glmptb \gnot{\hspace{-0,34cm}\psi_Y} \grmptb \gnl
\gvac{2} \glmptb \gnot{\hspace{-0,34cm}\phi_X} \grmptb \gcl{1} \glmptb \gnot{ \hspace{0,4cm} j^{-1}(\gamma^{-1}\beta^{-1}\gamma)} \gcmp \gcmp \grmp \gnl
\gvac{2} \gob{1}{F\s''} \gob{1}{X} \gob{1}{Y} \gob{1}{F}
\gend=
$$

$$
\stackrel{\equref{l8}}{=}
\gbeg{9}{8}
\gvac{1} \got{1}{F\s''} \got{1}{X} \got{1}{Y} \got{1}{F} \gnl
\glmp \gnot{\hspace{-0,34cm}j\s'(\gamma)} \grmptb \gcl{1} \gcl{1} \gcl{1} \gnl
\gvac{1} \glmptb \gnot{\hspace{-0,34cm}\psi_X} \grmptb \glmptb \gnot{\hspace{-0,34cm}\phi_Y} \grmptb \gnl
\gvac{1} \gcl{1} \glmptb \gnot{\hspace{-0,34cm}\lambda'_\beta} \grmptb \gcn{2}{2}{1}{3} \gnl
\gvac{1} \glmptb \gnot{\hspace{-0,34cm}\phi_X} \grmptb\glmpt \gnot{\hspace{-0,34cm}\gamma^{-1}\beta} \grmpb \gnl
\gvac{1} \gcl{2} \gcl{2} \gvac{1} \glmptb \gnot{\hspace{-0,34cm}\psi_Y} \grmptb \gnl
\gvac{4}  \gcl{1} \glmptb \gnot{ \hspace{0,4cm} j^{-1}(\gamma^{-1}\beta^{-1}\gamma)} \gcmp \gcmp \grmp \gnl
\gvac{1} \gob{1}{F\s''} \gob{1}{X} \gvac{1} \gob{1}{Y} \gob{1}{F}
\gend\stackrel{X}{=}
\gbeg{8}{9}
\gvac{1} \got{1}{F\s''} \got{1}{X} \got{1}{Y} \got{1}{F} \gnl
\glmp \gnot{\hspace{-0,34cm}j\s'(\gamma)} \grmptb \gcl{1} \glmptb \gnot{\hspace{-0,34cm}\phi_Y} \grmptb \gnl
\gvac{1} \gcl{1} \glmptb \gnot{\hspace{-0,34cm}\phi_X} \grmptb \gcl{1} \gnl
\gvac{1} \glmptb \gnot{\hspace{-0,34cm}\lambda''_\alpha } \grmptb \gcl{1} \gcn{2}{2}{1}{3} \gnl
\gvac{1} \gcl{2} \glmptb \gnot{\hspace{-0,34cm}\psi_X} \grmptb \gnl
\gvac{2} \gcl{1} \glmpt \gnot{\hspace{-0,34cm}\gamma^{-1}\beta} \grmpb \gcl{1} \gnl
\gvac{1} \gcl{2} \gcl{2} \gvac{1} \glmptb \gnot{\hspace{-0,34cm}\psi_Y} \grmptb \gnl
\gvac{4}  \gcl{1} \glmptb \gnot{ \hspace{0,4cm} j^{-1}(\gamma^{-1}\beta^{-1}\gamma)} \gcmp \gcmp \grmp \gnl
\gvac{1} \gob{1}{F\s''} \gob{1}{X} \gvac{1} \gob{1}{Y} \gob{1}{F}
\gend=
\gbeg{8}{11}
\gvac{1} \got{1}{F\s''} \got{1}{X} \got{1}{Y} \got{1}{F} \gnl
\gvac{1} \gcl{1} \gcl{1} \glmptb \gnot{\hspace{-0,34cm}\phi_Y} \grmptb \gnl
\glmp \gnot{\hspace{-0,34cm}j\s'(\gamma)} \grmptb \glmptb \gnot{\hspace{-0,34cm}\phi_X} \grmptb \gcl{1} \gnl
\gvac{1} \glmptb \gnot{\hspace{-0,34cm}\lambda''_\alpha } \grmptb \gcn{1}{1}{1}{3} \gcn{2}{2}{1}{5} \gnl
\gvac{1} \gcl{6} \glmptb \gnot{\hspace{-0,34cm} \cancel{ j\s'(\gamma^{-1})} } \grmp \gcl{2} \gnl
\gvac{2} \glmpt \gnot{\hspace{-0,34cm} \cancel{ j\s'(\gamma\w)} } \grmpb \gvac{2} \gcl{2} \gnl
\gvac{3} \glmptb \gnot{\hspace{-0,34cm}\psi_X} \grmptb \gnl
\gvac{3} \gcl{3} \glmpt \gnot{\hspace{-0,34cm}\gamma^{-1}\beta} \grmpb \gcl{1} \gnl
\gvac{5} \glmptb \gnot{\hspace{-0,34cm}\psi_Y} \grmptb \gnl
\gvac{5}  \gcl{1} \glmptb \gnot{ \hspace{0,4cm} j^{-1}(\gamma^{-1}\beta^{-1}\gamma)} \gcmp \gcmp \grmp \gnl
\gvac{1} \gob{1}{F\s''} \gvac{1} \gob{1}{X} \gvac{1} \gob{1}{Y} \gob{1}{F}
\gend\stackrel{\equref{l4}}{=}
\gbeg{3}{5}
\gvac{1} \got{1}{F\s''} \got{1}{X} \got{1}{F} \gnl
\gvac{1} \gcl{1} \glmptb \gnot{\hspace{-0,34cm}\phi_{XY}} \grmptb \gnl
\glmp \gnot{\lambda''_{\alpha j\s'(\gamma)}} \gcmptb \grmptb \gcl{1} \gnl
\gvac{1} \gcl{1} \glmptb \gnot{\hspace{-0,34cm}\psi_{XY}} \grmptb \gnl
\gvac{1} \gob{1}{F\s'} \gob{1}{X} \gob{1}{F.}
\gend
$$
\qed\end{proof}

\pagebreak

The above Proposition yields a functor for three bimonads $F, F\s', F\s''$ in $\K$ with isomorphic corresponding 0-automorphism groups $G,G', G''$: 
$$ \hspace{-1cm} c^{(\alpha, \beta),(\gamma, \delta)}_{F\s'', F\s', F}: \quad
\Bimnd^T(F\s'_\beta, F\s''_\alpha) \times \Bimnd^T(F_\delta,F\s'_\gamma) \to \Bimnd^T(F_{\delta\cdot j^{-1}(\gamma^{-1}\beta\gamma)},F\s''_{\alpha\cdot j\s'(\gamma)})$$
$$ \quad (X, \psi_X, \phi_X), \hspace{1cm} (Y, \psi_Y, \phi_Y) \hspace{0,5cm} \mapsto \hspace{1cm} (XY, \psi_{XY}, \phi_{XY})$$
where $\psi_{XY}$ and $\phi_{XY}$ are given by \equref{psi_YX}, \equref{phi_YX}, and $j:G\to G'$ and $j\s': G'\to G''$ are fusion maps. This functor is clearly well-defined on morphisms. 
Mind the notation: as indicated in \equref{hom cat bim Tur}, the $(\alpha, \beta)$-component of the category $\Bimnd^T(F, F\s')$ we write as $\Bimnd^T(F_\beta,F\s'_\alpha)$. 
Then the component of the codomain of the functor $c^{(\alpha, \beta),(\gamma, \delta)}_{F\s'', F\s', F}$ is the transitive product \equref{trans product} of $(\alpha, \beta)\in G''\times G'$ 
and $(\gamma, \delta)\in G'\times G$.

\begin{lma}
Let $(X, \psi_X, \phi_X) \in \Bimnd^T(F\s''_\beta, F\s'''_\alpha), (Y, \psi_Y, \phi_Y)\in \Bimnd^T(F\s'_\delta, F\s''_\gamma)$ and 
$(Z, \psi_Z, \phi_Z)\in \Bimnd^T(F_\nu, F\s'_\mu)$ be composable 1-cells in $\Bimnd^T(\K)$. Then the corresponding objects $(XY)Z$ and $X(YZ)$ are 
equal in $\Bimnd^T(F_{\nu j^{-1}\left(\mu^{-1}\delta j\s'^{-1}(\gamma^{-1}\beta\gamma)\mu\right)}, F\s'''_{\alpha j\s''(\gamma) j\s''j\s'(\mu)})$.  
(Here \\ 
$(\alpha j\s''(\gamma) j\s''j\s'(\mu), \nu j^{-1}\left(\mu^{-1}\delta j\s'^{-1}(\gamma^{-1}\beta\gamma)\mu\right))=
(\alpha, \beta)*(\gamma, \delta)*(\mu,\nu)$ is the transitive product as in \equref{triple trans prod} with fusion maps 
$j:G\to G', j\s': G'\to G''$ and $j\s'':G''\to G'''$.)  
\end{lma}

\begin{proof}
As in the proof of the above Proposition, we only have to compare $\psi_{(XY)Z}$ and $\psi_{X(YZ)}$. Set $j\s^{13}=j\s''j\s'$ and $j\s^{02}=j\s'j$ as before, we find: 
$$\hspace{-1,6cm}
\psi_{(XY)Z}\stackrel{\equref{psi_XY real}}{\stackrel{XY,Z}{=}}
\gbeg{12}{8}
\gvac{1} \got{1}{F\s'''} \got{1}{XY} \got{7}{Z} \gnl
\glmp \gnot{\hspace{-0,34cm}j\s^{13}(\mu)} \grmptb \gcl{1} \gvac{3} \gcl{4} \gnl
\gvac{1} \glmptb \gnot{\hspace{-0,34cm}\psi_{XY}} \grmptb \gnl
\gvac{1} \gcl{4} \glmpt \gnot{\hspace{0,34cm}j\s'^{-1}(\gamma^{-1}\beta\gamma)} \gcmp \gcmp \grmpb \gnl
\gvac{4} \glmp \gnot{\hspace{-0,34cm}\mu^{-1}\delta} \grmptb \gnl
\gvac{5} \glmptb \gnot{\hspace{-0,34cm}\psi_Z} \grmptb \gnl
\gvac{5} \gcl{1} \glmpt \gnot{ \hspace{2,2cm} j^{-1}(\mu^{-1}j\s'^{-1}(\gamma^{-1}\beta^{-1}\gamma)\delta^{-1}\mu)} \gcmpb \gcmp \gcmp \gcmp \gcmp \gcmp \grmp \gnl
\gob{3}{XY} \gvac{2} \gob{1}{Z} \gob{3}{F}
\gend\stackrel{\equref{psi_XY real}}{\stackrel{X,Y}{=}}
\gbeg{9}{12}
\gvac{1} \got{1}{F\s'''} \got{1}{X} \got{3}{Y} \got{5}{Z} \gnl
\glmp \gnot{\hspace{-0,34cm}j\s^{13}(\mu)} \grmptb \gcl{2} \gvac{1} \gcl{4} \gvac{3} \gcl{8} \gnl
\glmp \gnot{\hspace{-0,34cm}j\s''(\gamma)} \grmptb \gcl{1} \gvac{1} \gcl{3} \gnl
\gvac{1} \glmptb \gnot{\hspace{-0,34cm}\psi_X} \grmptb \gnl
\gvac{1} \gcl{7} \glmpt \gnot{\hspace{-0,34cm} \gamma^{\w-1}\beta} \grmpb \gnl
\gvac{3} \glmptb \gnot{\hspace{-0,34cm} \psi_Y} \grmptb \gnl
\gvac{3} \gcl{5} \glmpt \gnot{ \hspace{0,4cm} \cancel{ j\s'^{-1}\w(\gamma^{-1}\beta^{-1}\gamma)} } \gcmpb \gcmp \grmp \gnl
\gvac{4} \glmpt \gnot{\hspace{0,34cm} \cancel{ j\s'^{-1}(\gamma^{-1}\beta\gamma)} } \gcmp \gcmp \grmpb \gnl
\gvac{6} \glmp \gnot{\hspace{-0,34cm}\mu^{-1}\delta} \grmptb \gnl
\gvac{7} \glmptb \gnot{\hspace{-0,34cm}\psi_Z} \grmptb \gnl
\gvac{7} \gcl{1} \glmpt \gnot{ \hspace{2,2cm} j^{-1}(\mu^{-1}j\s'^{-1}(\gamma^{-1}\beta^{-1}\gamma)\delta^{-1}\mu)} \gcmpb \gcmp \gcmp \gcmp \gcmp \gcmp \grmp \gnl
\gob{3}{X} \gob{1}{Y} \gvac{3} \gob{1}{Z} \gob{3}{F}
\gend
$$
and 
$$ \hspace{-1,8cm} \psi_{X(YZ)}\stackrel{\equref{psi_XY real}}{\stackrel{X,YZ}{=}}
\gbeg{13}{7}
\gvac{3} \got{1}{F\s'''} \got{1}{X} \got{7}{YZ} \gnl
\glmp \gnot{\hspace{0,34cm} j\s''(\gamma j\s'(\mu))} \gcmp \gcmp \grmptb \gcl{1} \gvac{3} \gcl{3} \gnl
\gvac{3} \glmptb \gnot{\hspace{-0,34cm}\psi_X} \grmptb \gnl
\gvac{3} \gcl{3} \glmpt \gnot{\hspace{0,34cm} j\s'(\mu^{-1}) \gamma^{-1}\beta} \gcmp \gcmp \grmpb \gnl
\gvac{7} \glmptb \gnot{\hspace{-0,34cm}\psi_{YZ}} \grmptb \gnl
\gvac{7} \gcl{1} \glmpt \gnot{ \hspace{2,5cm} (j^{02})^{-1}(j\s'(\mu^{-1})\gamma^{-1}\beta^{-1}\gamma j\s'(\mu))} \gcmpb \gcmp \gcmp \gcmp \gcmp \gcmp \gcmp \grmp \gnl
\gvac{2} \gob{3}{XY} \gvac{2} \gob{1}{Z} \gob{3}{F}
\gend\stackrel{\equref{psi_XY real}}{\stackrel{X,Y}{=}}
\gbeg{15}{11}
\gvac{3} \got{1}{F\s'''} \got{1}{X} \gvac{3} \got{1}{Y} \gvac{1} \got{1}{Z} \gnl
\glmp \gnot{\hspace{0,34cm} j\s''(\gamma j\s'(\mu))} \gcmp \gcmp \grmptb \gcl{1} \gvac{3} \gcl{3} \gvac{1} \gcl{6} \gnl
\gvac{3} \glmptb \gnot{\hspace{-0,34cm}\psi_X} \grmptb \gnl
\gvac{3} \gcl{7} \glmpt \gnot{\hspace{0,34cm} \cancel{ j\s'(\mu^{-1}) } \gamma^{-1}\beta} \gcmp \gcmp \grmpb \gnl
\gvac{6} \glmp \gnot{\hspace{-0,34cm} \cancel{ j\s'(\mu)} } \grmptb \gcl{1} \gnl
\gvac{7} \glmptb \gnot{\hspace{-0,34cm}\psi_Y} \grmptb \gnl
\gvac{7} \gcl{4} \glmpt \gnot{\hspace{-0,34cm}\mu^{\w-1}\delta} \grmpb \gnl
\gvac{9} \glmptb \gnot{\hspace{-0,34cm}\psi_Z} \grmptb \gnl
\gvac{9} \gcl{1} \glmpt \gnot{ \hspace{0,8cm} j^{-1}(\cancel{ \mu^{-1} } \delta^{-1}\mu)} \gcmpb \gcmp \gcmp \grmp \gnl
\gvac{9} \gcl{1} \glmpt \gnot{ \hspace{2,6cm} j^{-1}(\mu^{-1})(j^{02})^{-1}(\gamma^{-1}\beta^{-1}\gamma) \cancel{ j^{-1}(\mu)} } \gcmpb \gcmp \gcmp \gcmp \gcmp \gcmp \gcmp \grmp \gnl
\gvac{2} \gob{3}{X} \gvac{2} \gob{1}{Y} \gob{3}{Z} \gob{1}{F.} 
\gend
$$
Comparing the right hand-sides in the above two expressions we see that they are the same, which completes the proof.
\qed\end{proof}

This shows that the composition of 1-cells in $\Bimnd^T(\K)$ over bimonads with isomorphic corresponding 0-automorphism groups is strictly associative. 
We will have that $\Bimnd^T(\K)$ is a Turaev 2-category once we prove that for every pair of 0-cells $(\A, F, (\lambda_{\alpha})_{\alpha\in G}), 
(\A', F\s', (\lambda_{\alpha}')_{\alpha\in G})$ in $\Bimnd^T(\K)$ there is a group map $\Fi^{F,F\s'}: G'\times G\to\Aut(\Bimnd^T(F, F\s')), 
(\alpha, \beta)\mapsto\Fi^{F,F\s'}_{(\alpha, \beta)}$ such that for every $(\alpha, \beta), (\gamma, \delta)\in G'\times G$ it is: 
$ \Fi^{F,F\s'}_{(\alpha, \beta)} \left( \Bimnd^T_{(\gamma, \delta)}(F, F\s') \right) = \Bimnd^T_{(\alpha, \beta)*(\gamma, \delta)*(\alpha, \beta)^{-1}}(F, F\s')$ 
and the conditions in points 6) $ii)$ and 7) $ii)$ of \deref{2-Tur} are satisfied. 
Recall that $\Bimnd^T_{(\gamma, \delta)}(F, F\s')=\Bimnd^T(F_\delta, F\s'_\gamma)$.

\bigskip

\begin{prop} \prlabel{fi gives a functor}
Let $F$ and $F\s'$ be bimonads in $\K$ with isomorphic 0-automorphism groups $G$ and $G'$, respectively, a fusion map $j: G\to G'$ and let $(\alpha, \beta), (\gamma, \delta)\in G'\times G$. 
The following defines an invertible functor: 
\begin{equation} \eqlabel{fi alfa-beta action}
\Fi^{F,F\s'}_{(\alpha, \beta)}: \Bimnd^T_{(\gamma, \delta)}(F, F\s') \to \Bimnd^T_{(\alpha, \beta)*(\gamma, \delta)*(\alpha, \beta)^{-1}}(F, F\s')
\end{equation}
$$(Y, \psi_Y, \phi_Y)\mapsto (Y, \psi_Y^{\gamma^{-1} j\s(\beta)\gamma\alpha^{-1}}, \phi_Y^{\beta j^{-1}(\alpha^{-1})})$$ 
where 
\begin{center} 
\begin{tabular}{p{8cm}p{0cm}p{7cm}}
\begin{equation} \eqlabel{G'-twisted2} \hspace{-3cm}
\psi_Y^{\gamma^{-1} j\s(\beta)\gamma\alpha^{-1}}=
\gbeg{5}{5}
\gvac{3} \got{1}{F\s'} \got{1}{Y} \gnl
\glmp \gnot{\hspace{0,5cm}\gamma^{-1} j\s(\beta)\gamma\alpha^{-1}} \gcmp \gcmp \grmptb \gcl{1} \gnl
\gvac{3} \glmptb \gnot{\hspace{-0,34cm}\psi_Y} \grmptb \gnl
\gvac{3} \gcl{1} \glmptb \gnot{\hspace{1,2cm}j^{-1}(\alpha\gamma^{-1})\beta^{-1}j^{-1}(\gamma)} \gcmp \gcmp \gcmp \gcmp \grmp \gnl
\gvac{3} \gob{1}{Y} \gob{1}{F}
\gend
\end{equation} & & 
\begin{equation} \eqlabel{fi-twisted}
\phi_Y^{\beta j^{-1}(\alpha^{-1})}=
\gbeg{5}{5}
\gvac{2} \got{1}{Y} \got{1}{F} \gnl
\gvac{2} \gcl{1} \glmptb \gnot{\beta j^{-1}(\alpha^{-1})} \gcmp \grmp \gnl
\gvac{2} \glmptb \gnot{\hspace{-0,34cm}\phi_Y} \grmptb \gnl
\glmp \gnot{\alpha j\s(\beta^{-1})} \gcmp \grmptb \gcl{1} \gnl
\gvac{2} \gob{1}{F\s'} \gob{1}{Y.}
\gend
\end{equation}
\end{tabular}
\end{center}
\end{prop}

\begin{proof}
We first note that $\psi_Y^{\gamma^{-1} j\s(\beta)\gamma\alpha^{-1}}$ and $\phi_Y^{\beta j^{-1}(\alpha^{-1})}$ are of the form: 
\begin{equation} \eqlabel{pregled psi-phi twisted}
\psi^\alpha=
\gbeg{3}{5}
\got{1}{F\s'} \got{1}{} \gnl
\gbmp{\alpha} \gcl{1} \gnl
\glmptb \gnot{\hspace{-0,34cm}\psi} \grmptb \gnl
\gcl{1} \glmptb \gnot{j^{-1}(\alpha^{-1})} \gcmp \grmp \gnl
\gob{1}{} \gob{1}{F}
\gend\hspace{2cm}\text{and}\hspace{2cm}
\phi^\beta=
\gbeg{2}{5}
\got{1}{} \got{1}{F} \gnl
\gvac{1} \gcl{1} \gbmp{\beta} \gnl
\gvac{1} \glmptb \gnot{\hspace{-0,34cm}\phi} \grmptb \gnl
\glmp \gnot{\hspace{-0,38cm} j\s(\beta^{-1})} \grmptb \gcl{1} \gnl
\gvac{2} \gob{1}{F\s'} \gob{1}{}
\gend
\end{equation}
for $\alpha\in G'$ and $\beta\in G$, respectively, for every 1-cell $(\A,F)\to(\A', F\s')$ in $\Mnd(\K)$ and $\Comnd(\K)$, respectively. Moreover, 
note that $\psi_Y^{\gamma^{-1} j\s(\beta)\gamma\alpha^{-1}}$ is as $\psi_X^\beta$ in \equref{G'-twisted}. Similarly as in \leref{twisted psi's} we have 
that $(Y, \phi_Y^{\beta j^{-1}(\alpha^{-1})})$ is a 1-cell in $\Comnd(\K)$. The assignment $\Fi^{F,F\s'}_{(\alpha, \beta)}$ is clearly invertible and it will be well-defined 
once we prove that the identity \equref{YD cond for alfa} holds. Then it will also be well-defined on morphisms and we will have the proof. We first observe 
the following: 
\begin{equation} \eqlabel{pomocni id}
\gbeg{6}{5}
\gvac{3} \got{1}{F\s'} \got{1}{F\s'} \gnl
\glmp \gnot{\hspace{0,5cm}\gamma^{-1} j\s(\beta)\gamma\alpha^{-1}} \gcmp \gcmp \grmptb \gcl{1} \gnl
\gvac{3} \glmptb \gnot{\hspace{-0,34cm}\lambda'_\gamma} \grmptb \gnl
\gvac{1} \glmp \gnot{\hspace{0,2cm} \alpha j\s(\beta^{-1})} \gcmp \grmptb \gcl{1} \gnl
\gvac{3} \gob{1}{F\s'} \gob{1}{F\s'}
\gend\stackrel{\equref{l6}}{=}
\gbeg{6}{5}
\gvac{2} \got{1}{F\s'} \got{1}{F\s'} \gnl
\glmp \gnot{\hspace{0,1cm} j\s(\beta)\gamma\alpha^{-1}} \gcmp  \grmptb \gcl{1} \gnl
\gvac{2} \glmptb \gnot{\hspace{-0,34cm}\lambda'} \grmptb \gnl
\glmp \gnot{\hspace{0,2cm} \alpha j\s(\beta^{-1})} \gcmp \grmptb \glmptb \gnot{\hspace{-0,34cm} \gamma^{-1}} \grmp \gnl
\gvac{2} \gob{1}{F\s'} \gob{1}{F\s'}
\gend\stackrel{\equref{l1}}{=}
\gbeg{7}{5}
\gvac{2} \got{1}{F\s'} \got{1}{F\s'} \gnl
\glmp \gnot{\hspace{0,1cm} j\s(\beta)\gamma\alpha^{-1}} \gcmp  \grmptb \glmptb \gnot{\hspace{0,2cm} \alpha j\s(\beta^{-1})} \gcmp \grmp \gnl
\gvac{2} \glmptb \gnot{\lambda'_{\alpha j\s(\beta^{-1})}} \gcmp \grmptb \gnl
\gvac{2} \gcl{1} \glmptb \gnot{\hspace{-0,34cm} \gamma^{-1}} \grmp \gnl
\gvac{2} \gob{1}{F\s'} \gob{1}{F\s'}
\gend\stackrel{\equref{l3'}}{=} 
\gbeg{6}{6}
\got{1}{F\s'} \gvac{1} \got{1}{F\s'} \gnl
\gcl{1} \glmp \gnot{\hspace{0,2cm} \alpha j\s(\beta^{-1})} \gcmptb \grmp \gnl
\glmptb \gnot{\lambda'_{\alpha \gamma\alpha^{-1}}} \gcmp \grmptb \gnl
\gcl{1} \glmp \gnot{ j\s(\beta)\gamma\alpha^{-1}} \gcmptb \grmp \gnl
\gcl{1} \glmpb \gnot{\hspace{-0,34cm} \gamma^{-1}} \grmpt \gnl
\gob{1}{F\s'} \gob{1}{F\s'}
\gend
\end{equation}
and 
\begin{equation} \eqlabel{conjug}
(\alpha, \beta)*(\gamma, \delta)*(\alpha, \beta)^{-1} = ( \alpha \gamma\alpha^{-1}, j^{-1}(\alpha)\beta^{-1}\delta j^{-1}(\gamma^{-1})\beta j^{-1}(\gamma\alpha^{-1})).
\end{equation} 
For simplicity reasons let us denote $\psi_{\crta Y}=\psi_Y^{\gamma^{-1} j\s(\beta)\gamma\alpha^{-1}}, \phi_{\crta Y}=\phi_Y^{\beta j^{-1}(\alpha^{-1})}$ 
and $\lambda_R= \lambda_{j^{-1}(\alpha)\beta^{-1}\delta j^{-1}(\gamma^{-1})\beta j^{-1}(\gamma\alpha^{-1})}$.
Then the left hand-side of \equref{twisted YD cond}  becomes: 
$$
\gbeg{3}{5}
\got{1}{F\s''} \got{1}{Y} \got{1}{F} \gnl
\glmptb \gnot{\hspace{-0,34cm}\psi_{\crta Y}} \grmptb \gcl{1} \gnl
\gcl{1} \glmptb \gnot{\hspace{-0,34cm}\lambda_R} \grmptb \gnl
\glmptb \gnot{\hspace{-0,34cm}\phi_{\crta Y}} \grmptb \gcl{1} \gnl
\gob{1}{F\s''} \gob{1}{Y} \gob{1}{F}
\gend \stackrel{\equref{l4}}{=}  
\gbeg{14}{11}
\gvac{3} \got{1}{F\s'} \got{1}{Y} \got{11}{F} \gnl
\glmp \gnot{\hspace{0,5cm}\gamma^{-1}j\s(\beta)\gamma\alpha^{-1}} \gcmp \gcmp \grmptb \gcl{1} \gvac{5} \gcl{4} \gnl
\gvac{3} \glmptb \gnot{\hspace{-0,34cm}\psi_Y} \grmptb \gnl
\gvac{3} \gcl{3} \glmpt \gnot{\hspace{1,2cm} \cancel{ j^{-1}(\alpha\gamma^{-1})\beta^{-1}j^{-1}(\gamma) } } \gcmp \gcmp \gcmpb \gcmp \grmp \gnl
\gvac{4} \glmp \gnot{\hspace{1,2cm} \cancel{ j^{-1}(\gamma^{-1})\beta j^{-1}(\gamma\alpha^{-1}) } } \gcmp \gcmp \gcmptb \gcmp \grmp \gnl
\gvac{7} \glmptb \gnot{\hspace{0,4cm} \lambda_{j^{-1}(\alpha)\beta^{-1}\delta}} \gcmp \gcmpb \grmpt \gnl
\gvac{3} \gcn{4}{2}{1}{5} \gcl{1} \gvac{1} \glmptb \gnot{\hspace{1,2cm}j^{-1}(\alpha\gamma^{-1})\beta^{-1}j^{-1}(\gamma)} \gcmp \gcmp \gcmp \gcmp \grmp  \gnl
\gvac{6} \glmpb \gnot{\s\beta j^{-1}(\alpha^{-1})} \gcmpt \grmp \gcl{3} \gnl
\gvac{5} \glmptb \gnot{\hspace{-0,34cm}\phi_Y} \grmptb \gnl
\gvac{3} \glmp \gnot{\alpha j\s(\beta^{-1})} \gcmp \grmptb \gcl{1} \gnl
\gvac{5} \gob{1}{F\s'} \gob{1}{Y} \gob{5}{F}
\gend\stackrel{\equref{l3'}}{=}
\gbeg{11}{8}
\gvac{3} \got{1}{F\s'} \got{1}{Y} \got{1}{F} \gnl
\glmp \gnot{\hspace{0,5cm}\gamma^{-1}j\s(\beta)\gamma\alpha^{-1}} \gcmp \gcmp \grmptb \gcl{1} \gcl{2} \gnl
\gvac{3} \glmptb \gnot{\hspace{-0,34cm}\psi_Y} \grmptb \gnl
\gvac{3} \gcl{2} \gcl{1} \glmptb \gnot{ \beta j^{-1}(\alpha^{-1}) } \gcmp \grmp \gnl
\gvac{4} \glmptb \gnot{\hspace{-0,34cm} \lambda_{\delta}} \grmptb \gnl
\gvac{3} \glmptb \gnot{\hspace{-0,34cm}\phi_Y} \grmptb  \glmptb \gnot{\hspace{1,2cm}j^{-1}(\alpha\gamma^{-1})\beta^{-1}j^{-1}(\gamma)} \gcmp \gcmp \gcmp \gcmp \grmp  \gnl
\gvac{1} \glmp \gnot{\alpha j\s(\beta^{-1})} \gcmp \grmptb \gcl{1} \gcl{1} \gnl
\gvac{3} \gob{1}{F\s'} \gob{1}{Y} \gob{1}{F}
\gend
$$

$$\stackrel{Y}{=}
\gbeg{11}{7}
\gvac{3} \got{1}{F\s'} \got{1}{Y} \got{1}{F} \gnl
\glmp \gnot{\hspace{0,5cm}\gamma^{-1}j\s(\beta)\gamma\alpha^{-1}} \gcmp \gcmp \grmptb \gcl{1} \glmptb \gnot{ \beta j^{-1}(\alpha^{-1}) } \gcmp \grmp \gnl
\gvac{3} \gcl{1} \glmptb \gnot{\hspace{-0,34cm}\phi_Y} \grmptb \gnl
\gvac{3} \glmptb \gnot{\hspace{-0,34cm}\lambda'_\gamma} \grmptb \gcl{1} \gnl
\gvac{1} \glmp \gnot{\alpha j\s(\beta^{-1})} \gcmp \grmptb \glmptb \gnot{\hspace{-0,34cm}\psi_Y} \grmptb \gnl
\gvac{3} \gcl{1}  \gcl{1} \glmptb \gnot{\hspace{1,2cm}j^{-1}(\alpha\gamma^{-1})\beta^{-1}j^{-1}(\gamma)} \gcmp \gcmp \gcmp \gcmp \grmp \gnl
\gvac{3} \gob{1}{F\s'} \gob{1}{X} \gob{1}{Y} \gob{1}{F}
\gend\stackrel{\equref{pomocni id}}{=}
\gbeg{8}{10}
\got{1}{F\s'} \gvac{2} \got{1}{Y} \got{1}{F} \gnl
\gcl{1} \gvac{2} \gcl{1} \glmptb \gnot{ \beta j^{-1}(\alpha^{-1}) } \gcmp \grmp \gnl
\gcl{1} \gvac{2} \glmptb \gnot{\hspace{-0,34cm}\phi_Y} \grmptb \gnl
\gcl{1} \glmp \gnot{\hspace{0,2cm} \alpha j\s(\beta^{-1})} \gcmpb \grmpt \gcl{3} \gnl
\glmptb \gnot{\lambda'_{\alpha \gamma\alpha^{-1}}} \gcmp \grmptb \gnl
\gcl{1} \glmp \gnot{ j\s(\beta)\gamma\alpha^{-1}} \gcmptb \grmp \gnl
\gcl{3} \glmp \gnot{\hspace{-0,34cm} \gamma^{-1}} \grmptb \gcn{1}{1}{3}{1} \gnl 
\gvac{2} \glmptb \gnot{\hspace{-0,34cm}\psi_Y} \grmptb \gnl
\gvac{2}  \gcl{1} \glmptb \gnot{\hspace{1,2cm}j^{-1}(\alpha\gamma^{-1})\beta^{-1}j^{-1}(\gamma)} \gcmp \gcmp \gcmp \gcmp \grmp \gnl
\gob{1}{F\s'} \gob{1}{} \gob{1}{Y} \gob{1}{F}
\gend=
\gbeg{4}{5}
\gvac{1} \got{1}{F\s'} \got{1}{Y} \got{1}{F} \gnl
\gvac{1} \gcl{1} \glmptb \gnot{\hspace{-0,34cm}\phi_{\crta Y}} \grmptb \gnl
\glmp \gnot{\lambda'_{\alpha \gamma\alpha^{-1}}} \gcmptb \grmptb \gcl{1} \gnl
\gvac{1} \gcl{1} \glmptb \gnot{\hspace{-0,34cm}\psi_{\crta Y}} \grmptb \gnl
\gvac{1} \gob{1}{F\s'} \gob{1}{X} \gob{1}{F.}
\gend
$$
\qed\end{proof}

\begin{cor}
The functor from \prref{fi gives a functor} induces an automorphism functor 
$$\Fi^{F,F\s'}_{(\alpha, \beta)} \in \Aut(\Bimnd^T(F, F\s')) .$$
\end{cor}

\bigskip

\begin{prop} \prlabel{Fi group map}
The assignment 
$$\Fi^{F,F\s'}: G'\times G\to\Aut(\Bimnd^T(F, F\s'))$$
$$ \hspace{-1,2cm} (\alpha, \beta)\mapsto\Fi^{F,F\s'}_{(\alpha, \beta)}$$
is a group map. 
\end{prop}

\begin{proof}
Take $(\alpha, \beta), (\gamma, \delta), (\mu, \nu) \in G'\times G$ and $Z\in\Bimnd^T_{(\mu,\nu)}(F, F\s')$, let us prove that 
$\Fi^{F,F\s'}_{(\alpha, \beta)*(\gamma,\delta)}(Z)=\Fi^{F,F\s'}_{(\alpha, \beta)}\comp \Fi^{F,F\s'}_{(\gamma,\delta)}(Z)$. For that purpose we recall from 
\equref{product G'G} that $(\alpha, \beta)*(\gamma,\delta)=(\alpha\gamma,\delta j^{-1}(\gamma^{-1})\beta j^{-1}(\gamma))$ and we observe that 
 $\Fi^{F,F\s'}_{(\gamma,\delta)}(Z)$ 
lies in the component $(\gamma, \delta)*(\mu, \nu)*(\gamma, \delta)^{-1} = ( \gamma \mu\gamma^{-1}, j^{-1}(\gamma)\delta^{-1}\nu j^{-1}(\mu^{-1})\delta j^{-1}(\mu\gamma^{-1}))$, 
see \equref{conjug}. Then it follows: 
$\Fi^{F,F\s'}_{(\alpha, \beta)*(\gamma,\delta)}(Z)=(Z, \psi_Z^\Omega, \phi_Z^\Sigma)$ where 
$$\Omega=\mu^{-1} j\s\left(\delta j^{-1}(\gamma^{-1})\beta j^{-1}(\gamma)\right)\mu\gamma^{-1}\alpha^{-1}=\mu^{-1} j\s(\delta)\gamma^{-1} j\s(\beta)\gamma\mu\gamma^{-1}\alpha^{-1}$$ 
and 
$$\Sigma=\delta j^{-1}(\gamma^{-1})\beta j^{-1}(\gamma) j^{-1}(\gamma^{-1}\alpha^{-1})=\delta j^{-1}(\gamma^{-1})\beta j^{-1}(\alpha^{-1}).$$
On the other hand we find: 
$$\begin{array}{rl}
\Fi^{F,F\s'}_{(\alpha, \beta)}\comp \Fi^{F,F\s'}_{(\gamma,\delta)}(Z) 
 \hskip-1em& =\Fi^{F,F\s'}_{(\alpha, \beta)}(Z, \psi_Z^{\mu^{-1} j\s(\delta)\mu\gamma^{-1}}, \phi_Z^{\delta j^{-1}(\gamma^{-1})})\\
& =(Z, (\psi_Z^{\mu^{-1} j\s(\delta)\cancel{\mu}\cancel{\gamma^{-1}} })^{\cancel{\gamma}\cancel{\mu^{-1} } \gamma^{-1}j\s(\beta)\gamma\mu\gamma^{-1}\alpha^{-1}}, 
  (\phi_Z^{\delta j^{-1}(\gamma^{-1})})^{\beta j^{-1}(\alpha^{-1})} ) \\
	&=(Z, \psi_Z^\Omega, \phi_Z^\Sigma)\\
	&=\Fi^{F,F\s'}_{(\alpha, \beta)*(\gamma,\delta)}(Z)
\end{array}$$
where in the third equality we used that 
$\psi^{\alpha\alpha'}=(\psi^\alpha)^{\alpha'}$ and $\phi^{\beta\beta'}=(\phi^\beta)^{\beta'}$, for all $\alpha, \alpha'\in G'$ and $\beta, \beta'\in G'$, 
which is clear from \equref{pregled psi-phi twisted}. Moreover, from \equref{fi alfa-beta action} it is clear that $\Fi^{F,F\s'}_{(e_{G'}, e_G)}=\Id$, 
the identity functor on the category $\Bimnd^T_{(\gamma, \delta)}(F, F\s')$ for every $(\gamma, \delta)\in G'\times G$, and thus on the category $\Bimnd^T(F, F\s')$. 
\qed\end{proof}


We finish the proof that $\Bimnd^T(\K)$ is a Turaev 2-category by proving the following Proposition. Its proof is straightforward and 
technical, we include it for the record.

\begin{prop}
For every $(X, \psi_X, \phi_X) \in \Bimnd^T(F\s'_\beta, F\s''_\alpha)$ and $(Y, \psi_Y, \phi_Y)\in \Bimnd^T(F_\delta,F\s'_\gamma)$ composable 1-cells in $\Bimnd^T(\K)$ 
and $(\mu, \nu)\in G''\times G$ it is
$$\Fi^{F, F\s''}_{(\mu, \nu)}(XY)=\Fi^{F\s',F\s''}_{(\mu, \nu)_{(1)}}(X)\cdot\Fi^{F,F\s'}_{(\mu, \nu)_{(2)}}(Y).$$
\end{prop}

\begin{proof}
To prove the result we note the following: 
from \equref{fi alfa-beta action} by \equref{trans product} we have 
\begin{equation} \eqlabel{on prod}
\Fi^{F, F\s''}_{(\mu, \nu)}(XY)=
(XY, \psi_{XY}^{j\s'(\gamma^{-1})\alpha^{-1} j^{02}(\nu)\alpha j\s'(\gamma)\mu^{-1}}, \phi_{XY}^{\nu (j^{02})^{-1}(\mu^{-1})}),
\end{equation}
and on the other hand, by \equref{pi for bim} it is $\pi((\mu, \nu))=\left((\mu, j(\nu)), (j\s'^{-1}(\mu), \nu)\right)$, hence 
we have: 
$$\Fi^{F\s',F\s''}_{(\mu, j(\nu))}(X)=(X, \psi_X^{\alpha^{-1} j\s'j\s(\nu)\alpha\mu^{-1}}, \phi_X^{j\s(\nu)(j\s')^{-1}(\mu^{-1})})
=:(X,\psi_{\crta X}, \phi_{\crta X})$$
and it lies in the component 
$$(\mu, j(\nu))(\alpha, \beta)(\mu, j(\nu))^{-1}=
\left(\mu\alpha\mu^{-1}, j\s'^{-1}(\mu)j\s(\nu^{-1})\beta j\s'^{-1}(\alpha^{-1})j\s(\nu) j\s'^{-1}(\alpha\mu^{-1})\right)=:
(\Omega^1, \Omega^2)=
\Omega,$$ 
and 
$$\Fi^{F,F\s'}_{((j\s')^{-1}(\mu), \nu)}(Y)=(Y, \psi_Y^{\gamma^{-1} j\s(\nu)\gamma (j\s')^{-1}(\mu^{-1})}, 
\phi_Y^{\nu j^{-1}((j\s')^{-1}(\mu^{-1}))})
=:(Y,\psi_{\crta Y}, \phi_{\crta Y})$$
and it lies in the component 
$$\begin{array}{rl}
(j\s'^{-1}(\mu), \nu) &\hskip-0.8em (\gamma, \delta)(j\s'^{-1}(\mu), \nu)^{-1}= \\
& \left(j\s'^{-1}(\mu)\gamma j\s'^{-1}(\mu^{-1}), 
j^{-1}(j\s'^{-1}(\mu))\nu^{-1}\delta j^{-1}(\gamma^{-1})\nu j^{-1}(\gamma j\s'^{-1}(\mu^{-1}))\right)=:
(\Sigma^1, \Sigma^2)=\Sigma.
\end{array}$$
Consequently, 
$$\Fi^{F\s',F\s''}_{(\mu, j(\nu))}(X) \cdot\Fi^{F,F\s'}_{((j\s')^{-1}(\mu), \nu)}(Y)=
(XY, \psi_{\crta X\crta Y}, \phi_{\crta X\crta Y})$$
where $\psi_{\crta X\crta Y}$ and $\phi_{\crta X\crta Y}$ are given by \equref{psi_XY real} and \equref{phi_YX}, respectively as follows: 
$$ 
\psi_{\crta X\crta Y}=
\gbeg{10}{7}
\gvac{1} \got{1}{F\s''} \got{1}{\crta X} \got{5}{\crta Y} \gnl
\glmp \gnot{\hspace{-0,36cm}j\s'(\Sigma^1)} \grmptb \gcl{1} \gvac{2} \gcl{3} \gnl
\gvac{1} \glmptb \gnot{\hspace{-0,34cm}\psi_{\crta X}} \grmptb \gnl
\gvac{1} \gcl{3} \glmpt \gnot{(\Sigma^1)^{\w-1}\Omega^2} \gcmp \grmpb \gnl
\gvac{4} \glmptb \gnot{\hspace{-0,34cm}\psi_{\crta Y}} \grmptb \gnl
\gvac{4} \gcl{1} \glmpt \gnot{ \hspace{1,3cm} j^{-1}((\Sigma^1)^{-1}(\Omega^2)^{-1}\Sigma^1)} \gcmpb \gcmp \gcmp \gcmp \grmp \gnl
\gob{3}{\crta X} \gob{3}{\crta Y} \gob{1}{F}
\gend \stackrel{\equref{G'-twisted}}{=}
\gbeg{14}{11}
\gvac{1} \got{1}{F\s''} \got{7}{\crta X} \got{5}{\crta Y} \gnl
\glmp \gnot{\hspace{-0,36cm}j\s'(\Sigma^1)} \grmptb \gvac{3} \gcl{2} \gvac{5} \gcl{6} \gnl
\glmp \gnot{\hspace{0,9cm} \alpha^{-1} j\s'j\s(\nu)\alpha\mu^{-1} } \gcmpt \gcmp \gcmp \grmpb \gnl
\gvac{4} \glmptb \gnot{\hspace{-0,34cm}\psi_X} \grmptb \gnl
\gvac{4} \gcl{6} \glmpt \gnot{\hspace{1,2cm} j^{-1}(\mu\alpha^{-1}j\s'j\s(\nu^{-1})\alpha)} \gcmp \gcmp \gcmpb \gcmp \grmp \gnl
\gvac{7} \glmp \gnot{(\Sigma^1)^{\w-1}\Omega^2} \gcmpt \grmpb \gnl
\gvac{5} \glmp \gnot{\hspace{1,3cm} \gamma^{-1} j\s(\nu)\gamma (j\s')^{-1}(\mu^{-1}) } \gcmp \gcmp\gcmp \gcmpt \grmpb \gnl
\gvac{10} \glmptb \gnot{\hspace{-0,34cm}\psi_Y} \grmptb \gnl
\gvac{10} \gcl{2} \glmpt \gnot{\hspace{1,7cm} j^{-1}\w((j\s')^{-1}\w(\mu)\gamma^{-1}j\s(\nu^{-1})\gamma) } \gcmp \gcmp\gcmp \gcmp \gcmp \grmp \gnl
\gvac{11} \glmpt \gnot{ \hspace{1,3cm} j^{-1}((\Sigma^1)^{-1}(\Omega^2)^{-1}\Sigma^1)} \gcmpb \gcmp \gcmp \gcmp \grmp \gnl
\gob{9}{\crta X} \gob{3}{\crta Y} \gob{1}{F}
\gend
$$

$$=
\gbeg{8}{9}
\gvac{3} \got{1}{F\s''} \got{1}{\crta X} \got{3}{\crta Y} \gnl
\gvac{1} \glmp \gnot{ j\s'(\gamma)\mu^{-1}} \gcmp\grmptb \gcl{2} \gvac{1} \gcl{4} \gnl
\glmp \gnot{\hspace{0,24cm}\alpha^{-1} j^{02}(\nu)\alpha } \gcmp \gcmp \grmptb \gnl
\gvac{3} \glmptb \gnot{\hspace{-0,34cm}\psi_X} \grmptb \gnl
\gvac{3} \gcl{4} \glmpt \gnot{\hspace{-0,34cm}\gamma^{\w-1}\beta} \grmpb \gnl
\gvac{5} \glmptb \gnot{\hspace{-0,34cm}\psi_Y} \grmptb \gnl
\gvac{5} \gcl{1} \glmpt \gnot{ \hspace{2,5cm} \nu^{-1}(j^{02})^{-1}(\alpha)j^{-1}(\beta^{-1})j^{-1}(\gamma)} \gcmpb \gcmp  \gcmp\gcmp\gcmp \gcmp \gcmp \grmp \gnl
\gvac{5} \gcl{1} \glmptb \gnot{ \hspace{2,5cm} (j^{02})^{-1}(\mu) j^{-1}(\gamma^{-1})(j^{02})^{-1}(\alpha^{-1})} \gcmp \gcmp \gcmp \gcmp\gcmp \gcmp \gcmp \grmp \gnl
\gvac{2} \gob{3}{\crta X} \gob{1}{\crta Y} \gob{1}{F}
\gend
$$

$$
\phi_{\crta X\crta Y}=
\gbeg{3}{4}
\got{1}{\crta X} \got{1}{\crta Y} \got{1}{F} \gnl
\gcl{1} \glmptb \gnot{\hspace{-0,34cm}\phi_{\crta Y}} \grmptb \gnl
\glmptb \gnot{\hspace{-0,34cm}\phi_{\crta X}} \grmptb \gcl{1} \gnl
\gob{1}{F\s''} \gob{1}{\crta X} \gob{1}{\crta Y} 
\gend \stackrel{\equref{fi-twisted}}{=}
\gbeg{15}{8}
\gvac{5} \got{1}{\crta X} \gvac{4} \got{1}{\crta Y} \got{1}{F} \gnl
\gvac{5} \gcl{4} \gvac{4} \gcl{1} \glmptb \gnot{ \hspace{0,9cm} \nu j^{-1}((j\s')^{-1}(\mu^{-1}))} \gcmp \gcmp \gcmp \grmp \gnl
\gvac{10} \glmptb \gnot{\hspace{-0,34cm}\phi_Y} \grmptb \gnl
\gvac{6} \glmpb \gnot{\hspace{0,8cm}\cancel{  (j\s')^{-1}(\mu) j\s(\nu^{-1})} } \gcmp \gcmp \gcmp \grmpt \gcl{4} \gnl 
\gvac{6} \glmptb \gnot{\hspace{0,8cm}\cancel{ j\s(\nu)(j\s')^{-1}(\mu^{-1})} }\gcmp \gcmp \gcmp \grmp \gnl
\gvac{5}\glmptb \gnot{\hspace{-0,34cm}\phi_X} \grmptb \gnl
\glmp \gnot{\hspace{1,2cm} j\s'((j\s')^{-1}(\mu) j\s(\nu^{-1})} \gcmp\gcmp \gcmp \gcmp \grmptb \gcl{1} \gnl
\gvac{5}\gob{1}{F\s''} \gob{1}{\crta X} \gob{9}{\crta Y} 
\gend=\phi_{XY}^{\nu (j^{02})^{-1}(\mu^{-1})}
$$ 
The last equality is obvious, it refers to the value in \equref{on prod}. On the other hand, for $\psi_{XY}^{j\s'(\gamma^{-1})\alpha^{-1} j^{02}(\nu)\alpha j\s'(\gamma)\mu^{-1}}$ 
from \equref{on prod} we find: 
$$\hspace{-4cm}
\psi_{XY}^{j\s'(\gamma^{-1})\alpha^{-1} j^{02}(\nu)\alpha j\s'(\gamma)\mu^{-1}}=
\gbeg{10}{7}
\gvac{4} \got{1}{F\s''} \got{1}{\crta X\crta Y} \gnl
\glmp \gnot{\hspace{0,8cm} j^{02}(\nu)\alpha j\s'(\gamma)\mu^{-1}} \gcmp\gcmp\gcmp\grmptb \gcl{2} \gnl
\gvac{1} \glmp \gnot{\hspace{0,24cm} j\s'(\gamma^{-1})\alpha^{-1}} \gcmp \gcmp \grmptb \gnl
\gvac{4} \glmptb \gnot{\hspace{-0,34cm}\psi_{XY}} \grmptb \gnl
\gvac{4} \gcl{1} \glmptb \gnot{ \hspace{1,7cm} (j^{02})^{-1}\left(j^{02}(\nu^{-1})\alpha j\s'(\gamma)\right)} \gcmp \gcmp\gcmp \gcmp \gcmp \grmp \gnl
\gvac{4} \gcl{1} \glmptb \gnot{ \hspace{1,7cm} (j^{02})^{-1}\left(\mu j\s'(\gamma^{-1})\alpha^{-1}\right)} \gcmp \gcmp\gcmp \gcmp \gcmp \grmp \gnl
\gvac{4} \gob{1}{\crta X\crta Y} \gob{1}{F}
\gend=
\gbeg{8}{10}
\gvac{4} \got{1}{F\s''} \got{1}{\crta X} \got{3}{\crta Y} \gnl
\glmp \gnot{\hspace{0,8cm} j^{02}(\nu)\alpha j\s'(\gamma)\mu^{-1}} \gcmp\gcmp\gcmp\grmptb \gcl{3} \gvac{1} \gcl{5} \gnl
\gvac{1} \glmp \gnot{\hspace{0,24cm} j\s'(\gamma^{-1})\alpha^{-1}} \gcmp \gcmp \grmptb \gnl
\gvac{3} \glmp \gnot{\hspace{-0,34cm} j\s'(\gamma)} \grmptb \gnl
\gvac{4} \glmptb \gnot{\hspace{-0,34cm}\psi_X} \grmptb \gnl
\gvac{4} \gcl{4} \glmpt \gnot{\hspace{-0,34cm}\gamma^{\w-1}\beta} \grmpb \gnl
\gvac{6} \glmptb \gnot{\hspace{-0,34cm}\psi_Y} \grmptb \gnl
\gvac{6} \gcl{1} \glmpt \gnot{ \hspace{0,5cm} j^{-1}(\gamma^{-1}\beta^{-1}\gamma)} \gcmpb \gcmp \grmp \gnl
\gvac{6} \gcl{1} \glmptb \gnot{ \hspace{3,4cm} (j^{02})^{-1}\left(\mu j\s'(\gamma^{-1})\alpha^{-1}(j^{02})(\nu^{-1})\alpha j\s'(\gamma)\right)} \gcmp \gcmp \gcmp \gcmp \gcmp\gcmp\gcmp \gcmp \gcmp \grmp \gnl
\gvac{3} \gob{3}{\crta X} \gob{1}{\crta Y} \gob{1}{F}
\gend
$$
which clearly is equal to the above $\psi_{\crta X\crta Y}$. 
\qed\end{proof}

\bigskip

The above construction is a Turaev 2-category for bimonads in $\K$. It is directly seen that it is a Turaev extension of the 2-category of bimonads $\Bimnd(\K)$ 
that we recalled in \seref{prelim}. 

\bigskip

\begin{ex}
When $\K=\dul{\C}$ is the 2-category induced by a monoidal category $\C$, a bimonad in $\dul{\C}$ is an algebra and a coalgebra $F$ in $\C$ together with a morphism 
$\lambda: F\ot F\to F\ot F$ so that the conditions \equref{bimonad} and \equref{monadic d.l.}--\equref{comonadic d.l.} with $F\s'=X=F$ hold. 
(As a matter of fact, we should consider horizontally symmetric diagrams to the latte ones, as the tensor product 
in monoidal categories is read from the left to the right, while the composition of 1-cells in bicategories is read the other way around.) 

The 0-automorphism group $G$ of a bimonad $F$ in $\dul{\C}$ is just the automorphism group of the algebra and coalgebra $F$ in $\C$. For two bimonads 
$(F,\lambda)$ and $(F\s', \lambda')$ in $\dul{\C}$ there is a category $\Bimnd(\dul{\C})(F,F\s')$ as described in \equref{hom cat bim Tur}, now understood in terms of $\C$. 
Fixing one 0-cell in $\Bimnd(\dul{\C})$, that is a bimonad $F$, we have a monoidal category $\Bimnd(\dul{\C})(F,F)={}^F_F\YD(\dul{\C})=
\displaystyle{\bigcup_{(\alpha,\beta)\in G\times G}^{\bullet} {}^F_F\YD(\dul{\C})_{(\alpha,\beta)}}$ whose objects are 
objects $X\in\C$ together with morphisms $\psi: F\ot X\to X\ot F$ and $\phi: X\ot F\to F\ot X$ in $\C$ satisfying \equref{monadic d.l.}--
\equref{psi-lambda-phi 2-bimonad} (with $F\s'=F$). 
\end{ex}

\begin{ex} \exlabel{brC}
If $\C$ in the above example is even braided with a braiding 
$\gbeg{2}{1}
\gbr \gnl
\gend$, a bimonad in $\dul{\C}$ is a bialgebra $F$ in $\C$ with $\lambda=
\gbeg{3}{5}
\got{2}{F} \got{1}{F} \gnl
\gcmu \gcl{1} \gnl
\gcl{1} \gbr  \gnl
\gmu \gcl{1} \gnl
\gob{2}{F} \gob{1}{F}
\gend$. The 0-automorphism group $G$ is then just the automorphism group of the bialgebra $F$ in $\C$. Given a bialgebra $F$ 
and a left $F$-module and $F$-comodule $X$ in $\C$, for the 2-cells {\em i.e.} morphisms $\psi, \phi$ and $\lambda_\alpha$ we may take the following:
$$
\psi=
\gbeg{4}{5}
\got{2}{F} \got{1}{X} \gnl
\gcmu \gcl{1} \gnl
\gcl{1} \gbr \gnl
\glm \gcl{1} \gnl
\gvac{1} \gob{1}{X} \gob{1}{F} 
\gend, \quad
\phi=
\gbeg{4}{5}
\gvac{1} \got{1}{X} \got{1}{F} \gnl
\glcm \gcl{1} \gnl
\gcl{1} \gbr \gnl
\gmu \gcl{1} \gnl
\gob{2}{F} \gob{1}{X} 
\gend, \quad
\lambda_\alpha=
\gbeg{4}{5}
\got{2}{F} \got{1}{F} \gnl
\gcmu \gcl{1} \gnl
\gbmp{\alpha} \gbr  \gnl
\gmu \gcl{1} \gnl
\gob{2}{F} \gob{1}{F}
\gend
$$ 
Then $\Bimnd(\dul{\C})(F,F)={}^F_F\YD(\dul{\C})$ coincides with the well-known monoidal category of Yetter-Drinfel`d modules in $\C$, 
see for example \cite{Besp, F-tr}. For $\alpha,\beta\in G(F)$ it is $X\in\Bimnd^T(F,F)_{(\alpha,\beta)}$ when the following holds: 
$$
\gbeg{3}{5}
\got{1}{F} \got{1}{X} \got{1}{F} \gnl
\glmptb \gnot{\hspace{-0,34cm}\psi} \grmptb \gcl{1} \gnl
\gcl{1} \glmptb \gnot{\hspace{-0,34cm}\lambda_\beta} \grmptb \gnl
\glmptb \gnot{\hspace{-0,34cm}\phi} \grmptb \gcl{1} \gnl
\gob{1}{F} \gob{1}{X} \gob{1}{F}
\gend=
\gbeg{5}{11}
\got{2}{F} \got{1}{X} \got{3}{F} \gnl
\gcmu \gcl{1} \gvac{1} \gcl{3} \gnl
\gcl{1} \gbr \gnl
\glm \gcn{1}{1}{1}{2} \gnl
\gvac{1} \gcl{1} \gcmu \gcl{1} \gnl
\gvac{1} \gcl{1} \gbmp{\beta} \gbr \gnl
\gvac{1} \gcl{1} \gmu \gcl{1} \gnl
\glcm \gcn{2}{1}{2}{1} \gcl{3} \gnl
\gcl{1} \gbr \gnl
\gmu \gcl{1} \gnl
\gob{2}{F} \gob{1}{X} \gob{3}{F}
\gend=
\gbeg{5}{11}
\got{2}{F} \gvac{1} \got{1}{X} \got{1}{F} \gnl
\gcn{2}{3}{2}{2} \glcm \gcl{1} \gnl
\gvac{2} \gcl{1} \gbr \gnl
\gvac{2} \gmu \gcl{4} \gnl
\gcmu \gcn{1}{1}{2}{1} \gnl
\gbmp{\alpha} \gbr \gnl
\gmu \gcn{1}{1}{1}{2}  \gnl
\gcn{2}{3}{2}{2} \gcmu \gcl{1} \gnl
\gvac{2} \gcl{1} \gbr \gnl
\gvac{2} \glm \gcl{1} \gnl
\gob{1}{\hspace{0,28cm} F} \gvac{2} \gob{1}{X} \gob{1}{F}
\gend=
\gbeg{3}{5}
\got{1}{F} \got{1}{X} \got{1}{F} \gnl
\gcl{1} \glmptb \gnot{\hspace{-0,34cm}\phi} \grmptb \gnl
\glmptb \gnot{\hspace{-0,34cm}\lambda_\alpha} \grmptb \gcl{1} \gnl
\gcl{1} \glmptb \gnot{\hspace{-0,34cm}\psi} \grmptb \gnl
\gob{1}{F} \gob{1}{X} \gob{1}{F.}
\gend
$$
\end{ex}

\begin{ex}
Applying to the last identity in the above Example $FX\eta_F$ from above and $FX\Epsilon_F$ from below one gets \equref{braided ab}. If the braiding of $\C$ is symmetric when acting 
between $F$ and $X$ and between $F$ and $F$ (that is, if on these objects the braiding fulfills 
$\gbeg{2}{1}
\gbr \gnl
\gend=
\gbeg{2}{1}
\gibr \gnl
\gend$), then it can be shown that \equref{braided ab} is equivalent to \equref{other YD}. 
\begin{center} 
\begin{tabular}{p{6.2cm}p{2cm}p{6.2cm}}
\begin{equation} \eqlabel{braided ab}
\gbeg{3}{9}
\got{2}{F} \got{1}{X}  \gnl
\gcmu \gcl{1} \gnl
\gcl{1} \gbr \gnl
\glm \gcl{1} \gnl
\gvac{1} \gcl{1} \gbmp{\beta} \gnl
\glcm \gcl{1} \gnl
\gcl{1} \gbr \gnl
\gmu \gcl{1} \gnl
\gob{2}{F} \gob{1}{X} 
\gend=
\gbeg{3}{5}
\got{2}{F} \gvac{1} \got{1}{X}  \gnl
\gcmu \glcm \gnl
\gbmp{\alpha} \gbr \gcl{1} \gnl
\gmu \glm  \gnl
\gob{1}{\hspace{0,28cm} F} \gvac{2} \gob{1}{X} 
\gend
\end{equation} & & \vspace{-0,6cm}
\begin{equation} \eqlabel{other YD}
\gbeg{4}{8}
\got{1}{F} \got{5}{X} \gnl
\gcn{2}{2}{1}{5} \gvac{1} \gcl{2} \gnl \gnl
\gvac{2} \glm \gnl
\gvac{2} \glcm \gnl
\gcn{2}{2}{5}{1} \gvac{1} \gcl{2} \gnl \gnl
\gob{1}{F} \gob{5}{X} 
\gend=
\gbeg{6}{9}
\gvac{1} \got{2}{F} \got{4}{X} \gnl
\gvac{1} \gcmu \gvac{1} \gcn{1}{1}{2}{2} \gnl
\gvac{1} \gcn{1}{1}{1}{0} \hspace{-0,21cm} \gcmu \glcm \gnl
\gvac{1} \gcl{1} \gbr \gcl{1} \gcl{4} \gnl
\gvac{1} \gbmp{\alpha} \gbmp{\s\beta^{-1}} \gbr \gnl
\gvac{1} \gcl{1} \gbr \gcl{1} \gnl
\gvac{1} \gcn{1}{1}{1}{2} \gmu \glm \gnl
\gvac{2} \hspace{-0,21cm} \gmu \gcn{1}{1}{4}{4} \gnl
\gvac{2} \gob{2}{F} \gob{4}{X} 
\gend
\end{equation}
\end{tabular}
\end{center} 
Then for $\alpha=\id_F$ the equation \equref{other YD} is precisely the identity (4.9) from \cite{F2}. Therein $F$ is a Hopf algebra in $\C$, hence $\beta^{-1}=S\beta$, where $S$ is the 
antipode of $F$, and the braiding satisfies the above symmetricity conditions. In \cite{F2} we constructed an object $F_{\beta}$ which satisfies \equref{other YD} with $X=F_{\beta}$ (and 
$\alpha=\id_F$) and we proved that there is an anti-group homomorphism  $\Aut(F)\to\BQ(\C;F), \beta\mapsto\End(F_\beta)$ from the Hopf automorphism group to the quantum Brauer group of $F$. 
This generalizes an anti-group homomorphism from \cite{CVZ} constructed in vector spaces, and the other results we developed in \cite{F2} generalized to braided monoidal categories 
the ones from \cite{VZ4}. 
\end{ex}

\begin{ex}
In the particular case when the braided monoidal category $\C$ from \exref{brC} is the category of vector spaces over a field $k$, bimonads in $\dul{\C}$ 
are bialgebras over $k$ and the underlying 2-category of the Turaev 2-category is the monoidal category of the classical Yetter-Drinfel`d modules over $k$. 
Our Turaev 2-category $\Bimnd^T(\dul{\C})$ recovers the Turaev category consisting of generalized Yetter-Drinfel`d modules constructed in \cite{PS}. 
\end{ex}

\begin{ex}
As indicated in \cite{PS}, in the context of the latter example, when $F$ is a Hopf algebra over $k$ and $(\alpha,\beta)=(S^2,\id_H)$, where $S$ is the antipode 
of $H$, a generalized $(\alpha, \beta)$-Yetter-Drinfeld module is an anti-Yetter-Drinfeld module introduced in \cite{H1, H2}. Anti-Yetter-Drinfeld modules 
emerged as coefficients for the cyclic cohomology of Hopf algebras, which was introduced by Connes and Moscovici in \cite{CM}. 
\end{ex}

\subsection{Pairs in involution}

Let $(\A,F)$ and $(\A',F\s')$ be bimonads in $\K$.
The identity 1-cell $\Id_{\A}$ is trivially a monad and a comonad 
and we can consider the monads of the 2-cells $F\to\Id_{\A}$ and $\Id_{\A}\to F$ in $\K$, which are indeed convolution algebras in the monoidal category $\K(\A,\A)$. 
We will denote by $*$ the convolution product. 
Let $(\Id_\A,f):(\A, F\s')\to(\A,\Id_\A)$ be a 1-cell in $\Mnd(\K)$ and $(\Id_\A,g):(\A, \Id_\A)\to(\A,F\s')$ a 1-cell in $\Comnd(\K)$, so that $f$ and $g$ are 
convolution invertible in their respective convolution algebras in $\K(\A,\A)$. Recall that the former mean that $f$ satisfies \equref{alg m} and $g$ 
satisfies \equref{coalg m}. 

Suppose that $G\iso G'$, take $\alpha\in G', \beta\in G$ and set $\tilde\beta:=j\beta j^{-1}\in G(F\s')$. 
We say that $(f,g)$ is a {\em pair in involution corresponding to $(\alpha, \beta)$} if 
$$\alpha=
g^{-1}f*\tilde\beta* gf^{-1} \quad\Leftrightarrow\quad \alpha*g^{-1}f=g^{-1}f*\tilde\beta 
\quad\Leftrightarrow\quad gf^{-1}*\alpha=\tilde\beta* gf^{-1}$$
holds. In string diagrams these conditions look as follows: 
$$
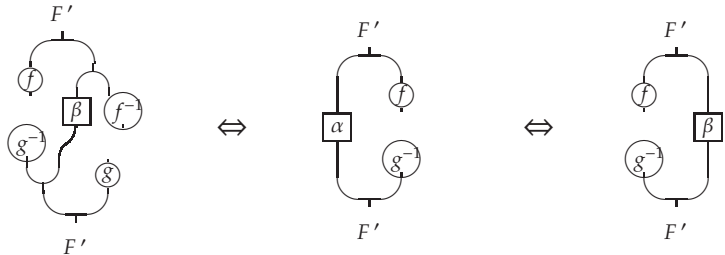
\gbeg{4}{8}
\gvac{1} \got{1}{F\s'} \gnl
\gwcm{3} \gnl
\gmp{f} \gvac{1} \hspace{-0,34cm} \gcmu \gnl
\gvac{1} \gvac{1} \gbmp{\beta}  \gvac{1} \hspace{-0,34cm}  \gelt{\hspace{0,14cm}f^{-1}} \gnl 
\gvac{1} \gelt{\hspace{0,14cm}g^{-1}} \gcn{1}{1}{2}{1}  \gnl
\gvac{1} \gmu \gvac{1} \hspace{-0,2cm} \gmp{g} \gnl
\gvac{2} \gwmu{3} \gnl
\gvac{3} \gob{1}{F\s'} 
\gend \qquad\Leftrightarrow\qquad 
\gbeg{4}{7}
\gvac{1} \got{1}{F\s'} \gnl
\gwcm{3} \gnl
\gcl{1}  \gvac{1} \gmp{f} \gnl
\gbmp{\alpha} \gnl
\gcl{1} \gvac{1} \gelt{\hspace{0,14cm}g^{-1}} \gnl
\gwmu{3} \gnl
\gvac{1} \gob{1}{F\s'} 
\gend \qquad\Leftrightarrow\qquad 
\gbeg{4}{7}
\gvac{1} \got{1}{F\s'} \gnl
\gwcm{3} \gnl
\gmp{f} \gvac{1} \gcl{1} \gnl
\gvac{2} \gbmp{\beta} \gnl
\gelt{\hspace{0,14cm}g^{-1}} \gvac{1} \gcl{1} \gnl
\gwmu{3} \gnl
\gvac{1} \gob{1}{F\s'} 
\gend
$$

The following is obvious: 

\begin{lma} \lelabel{obv pair}
Being $\Epsilon$ and $\eta$ the counit and the unit 2-cells of a bimonad $F$ in $\K$, the pair $(\Epsilon, \eta)$ is a pair in involution corresponding to 
$(\alpha, \alpha)$ for every $\alpha\in G$. 
\end{lma}

\medskip

The above definition is a 2-categorical analogy of the concept of {\em modular pair in involution} due to Connes and Moscovici.

\begin{lma} 
Let $(\A,F)$ be a monad and a comonad, $(F, \tau_{F,F})$ a 1-cell in $\Mnd(\K)$ and $\Comnd(\K)$ and let $\alpha\in\Aut_0(F)$. Then $(F, \lambda_\alpha)$, 
with $\lambda_\alpha$ defined by
\begin{equation} \eqlabel{lambda_alfa}
\lambda_\alpha=
\gbeg{4}{5}
\gvac{1} \got{1}{F} \got{3}{F} \gnl
\gwcm{3} \gcl{1} \gnl
\gbmp{\alpha}  \gvac{1} \glmptb \gnot{\hspace{-0,34cm}\tau_{F,F}} \grmptb \gnl
\gwmu{3} \gcl{1} \gnl
\gvac{1} \gob{1}{F} \gvac{1} \gob{2}{\hspace{-0,34cm}F}
\gend
\end{equation}
is a 1-cell both in $\Mnd(\K)$ and $\Comnd(\K)$ satisfying \equref{lambda alfa-beta}. 
\end{lma}

\begin{proof}
We observed in \cite{Femic6} that $(F, \lambda_e)$ is a 1-cell both in $\Mnd(\K)$ and $\Comnd(\K)$, here $e=id_F$. The proof of the claim for general $\alpha$ is direct. 
\qed\end{proof}

Assuming that for all $\alpha\in G'$ the 2-cells $\lambda_\alpha'$ are of the form \equref{lambda_alfa}, we have the following identities (where $\beta\in G$ and $\tilde\beta\in G'$ 
are as above): 
\begin{equation} \eqlabel{lambda beta calculo}
\gbeg{3}{7}
\got{1}{F\s'} \got{2}{} \got{1}{\hspace{-0,32cm}F\s'} \gnl
\gcl{2}  \gvac{1} \hspace{-0,34cm} \gmp{g} \gcl{1} \gnl
\gvac{2} \gmu \gnl
\gvac{1} \hspace{-0,2cm} \glmptb \gnot{\lambda_{\tilde\beta}'} \gcmp \grmptb \gnl
\gvac{1} \gcl{2} \gvac{1} \hspace{-0,34cm} \gcmu \gnl
\gvac{3} \gelt{\hspace{0,14cm}f^{-1}} \gcl{1} \gnl
\gvac{1} \gob{1}{\hspace{0,32cm}F\s'} \gvac{2} \gob{1}{F\s'} \gnl
\gend=
\gbeg{5}{8}
\gvac{1} \got{2}{F\s'} \got{2}{} \got{1}{\hspace{-0,32cm}F\s'} \gnl
\gvac{1} \gcn{2}{1}{2}{2}  \gvac{1} \hspace{-0,34cm} \gmp{g} \gcl{1} \gnl
\gvac{1} \gwcm{3} \gmu \gnl
\gvac{1} \gbmp{\tilde\beta} \gvac{1} \gcl{1} \gcn{1}{1}{2}{1}  \gnl
\gvac{1} \gcl{1} \gvac{1} \glmptb \gnot{\hspace{-0,34cm}\tau_{F\s',F\s'}} \grmptb \gnl
\gvac{1} \gwmu{3} \hspace{-0,22cm} \gcmu \gnl
\gvac{2} \gcn{2}{1}{2}{2} \gelt{\hspace{0,14cm}f^{-1}} \gcl{1} \gnl
\gvac{2} \gob{1}{\hspace{0,32cm}F\s'} \gvac{2} \gob{1}{F\s'} \gnl
\gend\stackrel{d.l.}{\stackrel{coass.}{\stackrel{ass.}{=}}}
\gbeg{5}{9}
\gvac{2} \got{1}{F\s'} \got{3}{F\s'} \gnl
\gvac{1} \gwcm{3} \gcl{3} \gnl
\gwcm{3} \gcl{2} \gnl
\gcl{1}  \gvac{1} \gelt{\hspace{0,14cm}f^{-1}} \gnl
\gbmp{\tilde\beta} \gvac{2} \glmptb \gnot{\hspace{-0,34cm}\tau_{F\s',F\s'}} \grmptb \gnl
\gcl{1} \gvac{1} \gmp{g} \gcl{2} \gcl{3} \gnl
\gwmu{3} \gnl
\gvac{1} \gwmu{3} \gnl
\gvac{2} \gob{1}{F\s'} \gob{3}{F\s'} 
\gend=\lambda_{\tilde\beta\s * gf^{-1}}'
\end{equation}
and similarly 
\begin{equation} \eqlabel{lambda alfa calculo}
\gbeg{4}{7}
\got{2}{F\s'} \got{1}{F\s'} \gnl
\gcmu  \gcl{2} \gnl
\gmp{f} \gcl{1} \gnl
\gvac{1} \glmptb \gnot{\hspace{-0,34cm}\lambda_{\tilde\beta}'} \grmptb \gnl
\gelt{\hspace{0,14cm}g^{-1}} \gcl{1} \gcl{2} \gnl
\gmu \gnl
\gob{2}{\hspace{0,32cm}F\s'} \gob{1}{F\s'} \gnl
\gend=
\gbeg{5}{7}
\gvac{1} \got{2}{F\s'} \got{1}{} \got{1}{F\s'} \gnl
\gwcm{3} \gcn{1}{2}{3}{3} \gnl
\gmp{f} \gwcm{3}  \gnl
\gvac{1} \gbmp{\tilde\beta}  \gvac{1} \glmptb \gnot{\hspace{-0,34cm}\tau_{F\s',F\s'}} \grmptb\gnl
\gelt{\hspace{0,14cm}g^{-1}} \gwmu{3} \gcl{2} \gnl
\gwmu{3}  \gnl
\gvac{1} \gob{1}{\hspace{0,32cm}F\s'} \gvac{2} \gob{1}{F\s'} \gnl
\gend\stackrel{coass.}{\stackrel{ass.}{=}}
\gbeg{5}{9}
\gvac{2} \got{1}{F\s'} \got{3}{F\s'} \gnl
\gvac{1} \gwcm{3} \gcl{3} \gnl
\gwcm{3} \gcl{2} \gnl
\gmp{f} \gvac{1} \gcl{1} \gnl
\gvac{2} \gbmp{\beta} \glmptb \gnot{\hspace{-0,34cm}\tau_{F\s',F\s'}} \grmptb \gnl
\gelt{\hspace{0,14cm}g^{-1}} \gvac{1} \gcl{1} \gcl{2} \gcl{3} \gnl
\gwmu{3} \gnl
\gvac{1} \gwmu{3} \gnl
\gvac{2} \gob{1}{F\s'} \gob{3}{F\s'} 
\gend=\lambda_{g^{-1}f*\tilde\beta}'=\lambda_{\alpha\s*g^{-1}f}' \stackrel{\equref{lambda beta calculo}}{=}
\gbeg{3}{7}
\got{1}{F\s'} \got{2}{} \got{1}{\hspace{-0,32cm}F\s'} \gnl
\gcl{2}  \gvac{1} \hspace{-0,34cm} \gelt{\hspace{0,14cm}g^{-1}} \gcl{1} \gnl
\gvac{2} \gmu \gnl
\gvac{1} \hspace{-0,2cm} \glmptb \gnot{\lambda_{\alpha}'} \gcmp \grmptb \gnl
\gvac{1} \gcl{2} \gvac{1} \hspace{-0,34cm} \gcmu \gnl
\gvac{3} \gmp{f} \gcl{1} \gnl
\gvac{1} \gob{1}{\hspace{0,32cm}F\s'} \gvac{2} \gob{1}{F\s'} \gnl
\gend
\end{equation}

\begin{thm}
Let $F$ and $F\s'$ be bimonads in $\K$ with isomorphic 0-automorphism groups $G$ and $G'$, respectively, let $\alpha\in G', \beta\in G$ and set 
$\tilde\beta:=j\beta j^{-1}\in G'$. Suppose that $(f,g)$ is a pair in involution corresponding to $(\alpha, \beta)$, then there is an 
isomorphism of categories 
$$\Bimnd^T(F_\beta, F\s'_\alpha)\iso\Bimnd(F, F\s').$$
\end{thm}

\begin{proof}
We define the functors $\F: \Bimnd^T(F_\beta, F\s'_\alpha)\to\Bimnd(F, F\s')$ and  $\G:\Bimnd(F, F\s')\to\Bimnd^T(F_\beta, F\s'_\alpha)$ as follows. 
For $(X,\psi_X,\phi_X)\in\Bimnd^T(F_\beta, F\s'_\alpha)$ set $\F\left((X,\psi_X,\phi_X)\right)=(\crta X,\psi_{\crta X},\phi_{\crta X})$, and for 
for $(Y,\psi_Y,\phi_Y)\in\Bimnd(F, F\s')$ set $\G\left((Y,\psi_Y,\phi_Y)\right)=(\dcr Y,\psi_{\dcr Y},\phi_{\dcr Y})$, where $\crta X=X, \dcr Y=Y$ and 
$\psi_{\crta X}, \phi_{\crta X}, \psi_{\dcr Y}$ and $\phi_{\dcr Y}$ are given as below: 
$$
\psi_{\crta X}=
\gbeg{5}{6}
\gvac{1} \got{1}{F\s'} \got{3}{X} \gnl
\glmp \gnot{\hspace{-0,34cm}\tilde\beta^{-1}} \grmptb \gvac{1} \gcl{2} \gnl
\gwcm{3} \gnl
\gelt{\hspace{0,14cm}f^{-1}} \gvac{1} \glmptb \gnot{\hspace{-0,34cm}\psi_X} \grmptb \gnl 
\gvac{2} \gcl{1} \gbmp{\beta} \gnl
\gvac{2} \gob{1}{X} \gob{1}{F}
\gend, \qquad
\phi_{\crta X}=
\gbeg{4}{4}
\gvac{1} \got{1}{X} \got{1}{F} \gnl
\gmp{g} \glmptb \gnot{\hspace{-0,34cm}\phi_X} \grmptb \gnl
\gmu \gcl{1} \gnl
\gob{2}{F\s'} \gob{1}{X}
\gend, \qquad
\psi_{\dcr Y}=
\gbeg{5}{6}
 \got{2}{F\s'} \got{1}{Y} \gnl
\gcmu \gcl{2} \gnl
\gmp{f}  \gbmp{\tilde\beta} \gnl
\gvac{1} \glmptb \gnot{\hspace{-0,34cm}\psi_Y} \grmptb \gnl 
\gvac{1} \gcl{1} \glmptb \gnot{\hspace{-0,34cm}\beta^{-1}} \grmp \gnl
\gvac{1} \gob{1}{Y} \gob{1}{F} 
\gend, \qquad
\phi_{\dcr Y}=
\gbeg{4}{4}
\gvac{1} \got{1}{Y} \got{1}{F} \gnl
\gelt{g^{-1}} \glmptb \gnot{\hspace{-0,34cm}\phi_Y} \grmptb \gnl
\gmu \gcl{1} \gnl
\gob{2}{F\s'} \gob{1}{Y}
\gend.$$
Once we prove that the functors $\F$ and $\G$ are well-defined it is easily seen that they are inverse of each other. Take
$(X,\psi_X,\phi_X)\in\Bimnd^T(F_\beta, F\s'_\alpha)$, let us prove the strong Yetter-Drinfel`d condition for 
$(\crta X,\psi_{\crta X},\phi_{\crta X})$:
$$
\gbeg{3}{5}
\got{1}{F\s'} \got{1}{\crta X} \got{1}{F} \gnl
\glmptb \gnot{\hspace{-0,34cm}\psi_{\crta X}} \grmptb \gcl{1} \gnl
\gcl{1} \glmptb \gnot{\hspace{-0,34cm}\lambda} \grmptb \gnl
\glmptb \gnot{\hspace{-0,34cm}\phi_{\crta X}} \grmptb \gcl{1} \gnl
\gob{1}{F\s'} \gob{1}{\crta X} \gob{1}{F}
\gend=
\gbeg{5}{9}
\gvac{1} \got{1}{F\s'} \gvac{1} \got{1}{X} \got{1}{F} \gnl
\glmp \gnot{\hspace{-0,34cm}\tilde\beta^{-1}} \grmptb \gvac{1} \gcl{2} \gcl{4} \gnl
\gwcm{3} \gnl
\gelt{\hspace{0,14cm}f^{-1}} \gvac{1} \glmptb \gnot{\hspace{-0,34cm}\psi_X} \grmptb \gnl 
\gvac{2} \gcl{2} \gbmp{\beta} \gnl
\gvac{3} \glmptb \gnot{\hspace{-0,34cm}\lambda} \grmptb \gnl
\gvac{1} \gmp{g} \glmptb \gnot{\hspace{-0,34cm}\phi_X} \grmptb \gcl{2} \gnl
\gvac{1} \gmu \gcl{1} \gnl
\gvac{1} \gob{2}{F\s'} \gob{1}{X} \gob{1}{F}
\gend\stackrel{\equref{l5}}{=}
\gbeg{6}{9}
\gvac{1} \got{1}{F\s'} \gvac{1} \got{1}{X} \got{1}{F} \gnl
\glmp \gnot{\hspace{-0,34cm}\tilde\beta^{-1}} \grmptb \gvac{1} \gcl{2} \gcl{3} \gnl
\gwcm{3} \gnl
\gelt{\hspace{0,14cm}f^{-1}} \gvac{1} \glmptb \gnot{\hspace{-0,34cm}\psi_X} \grmptb \gnl 
\gvac{2} \gcl{2} \glmptb \gnot{\hspace{-0,34cm}\lambda_\beta} \grmptb \gnl
\gvac{3} \gcl{1} \gbmp{\beta} \gnl
\gvac{1} \gmp{g} \glmptb \gnot{\hspace{-0,34cm}\phi_X} \grmptb \gcl{2} \gnl
\gvac{1} \gmu \gcl{1} \gnl
\gvac{1} \gob{2}{F\s'} \gob{1}{X} \gob{1}{F}
\gend\stackrel{\equref{twisted YD cond}}{=}
\gbeg{5}{7}
\gvac{1} \got{1}{F\s'} \gvac{1} \got{1}{X} \got{1}{F} \gnl
\glmp \gnot{\hspace{-0,34cm}\tilde\beta^{-1}} \grmptb \gvac{1} \gcl{1} \gcl{1} \gnl
\gwcm{3} \glmptb \gnot{\hspace{-0,34cm}\phi_X} \grmptb \gnl
\gelt{\hspace{0,14cm}f^{-1}} \gvac{1} \glmptb \gnot{\hspace{-0,34cm}\lambda_\alpha'} \grmptb \gcl{1} \gnl 
\gvac{1} \gmp{g} \gcl{1} \glmptb \gnot{\hspace{-0,34cm}\psi_X} \grmptb \gnl
\gvac{1} \gmu \gcl{1} \gbmp{\beta} \gnl
\gvac{1} \gob{2}{F\s'} \gob{1}{X} \gob{1}{F}
\gend\stackrel{*}{=}
\gbeg{5}{9}
\got{1}{F\s'} \gvac{1} \got{1}{X} \got{1}{F} \gnl
\gcl{1} \gmp{g} \glmptb \gnot{\hspace{-0,34cm}\phi_X} \grmptb \gnl
\gcn{1}{1}{1}{2} \gmu \gcl{1} \gnl
\gvac{1} \hspace{-0,34cm} \glmptb \gnot{\hspace{-0,34cm}\lambda'} \grmptb \gcn{1}{2}{2}{5} \gnl
\gvac{1} \gcl{4} \glmpt \gnot{\hspace{-0,34cm}\tilde\beta^{-1}} \grmpb \gnl
\gvac{2} \gwcm{3} \gcl{1} \gnl
\gvac{2} \gelt{\hspace{0,14cm}f^{-1}} \gvac{1} \glmptb \gnot{\hspace{-0,34cm}\psi_X} \grmptb \gnl 
\gvac{4} \gcl{1} \gbmp{\beta} \gnl
\gvac{1} \gob{1}{F\s'} \gvac{2} \gob{1}{X} \gob{1}{F}
\gend=
\gbeg{3}{5}
\got{1}{F\s'} \got{1}{\crta X} \got{1}{F} \gnl
\gcl{1} \glmptb \gnot{\hspace{-0,34cm}\phi_{\crta X}} \grmptb \gnl
\glmptb \gnot{\hspace{-0,34cm}\lambda'} \grmptb \gcl{1} \gnl
\gcl{1} \glmptb \gnot{\hspace{-0,34cm}\psi_{\crta X}} \grmptb \gnl
\gob{1}{F\s'} \gob{1}{\crta X} \gob{1}{F}
\gend
$$
the equality $*$ holds if and only if the following holds true (compose the equality $*$ from above with $\tilde\beta\times\id_{XF}$ and from below with $\id_{F\s'X}\times \beta^{-1}$, 
and then multiply the obtained expression from the left in the convolution algebra of 2-cells $F\s'\to F\s'$ by $g^{-1}f$): 
$$
\gbeg{3}{5}
\got{1}{F\s'} \got{1}{X} \got{1}{F} \gnl
\gcl{1} \glmptb \gnot{\hspace{-0,34cm}\phi_X} \grmptb \gnl
\glmptb \gnot{\hspace{-0,34cm}\lambda'_\alpha} \grmptb \gcl{1} \gnl
\gcl{1} \glmptb \gnot{\hspace{-0,34cm}\psi_X} \grmptb \gnl
\gob{1}{F\s'} \gob{1}{X} \gob{1}{F}
\gend=
\gbeg{5}{9}
\got{2}{F\s'} \gvac{1} \got{1}{X} \got{1}{F} \gnl\gcmu \gvac{1} \gcl{1} \gcl{1} \gnl
\gmp{f} \gbmp{\tilde\beta}  \gmp{g} \glmptb \gnot{\hspace{-0,34cm}\phi_X} \grmptb \gnl
\gvac{1} \gcn{1}{1}{1}{2} \gmu \gcl{1} \gnl
\gvac{2} \hspace{-0,34cm} \glmptb \gnot{\hspace{-0,34cm}\lambda'} \grmptb \gcn{1}{2}{2}{5} \gnl
\gvac{1} \gelt{\hspace{0,14cm}g^{-1}}  \gcl{1} \glmpt \gnot{\hspace{-0,34cm}\tilde\beta^{-1}} \grmpb \gnl
\gvac{1} \gmu \gwcm{3} \gcl{1} \gnl
\gvac{1} \gcn{2}{1}{2}{2} \gelt{\hspace{0,14cm}f^{-1}} \gvac{1} \glmptb \gnot{\hspace{-0,34cm}\psi_X} \grmptb \gnl 
\gvac{1} \gob{2}{F\s'} \gvac{2} \gob{1}{X} \gob{1}{F}
\gend
$$
This is true because
$$
\gbeg{5}{9}
\got{2}{F\s'} \gvac{1} \got{1}{F} \gnl\gcmu \gvac{1} \gcl{2} \gnl
\gmp{f} \gbmp{\tilde\beta}  \gmp{g}  \gnl
\gvac{1} \gcn{1}{1}{1}{2} \gmu \gnl
\gvac{2} \hspace{-0,24cm} \glmptb \gnot{\hspace{-0,34cm}\lambda'} \grmptb \gnl
\gvac{1} \gelt{\hspace{0,14cm}g^{-1}}  \gcl{1} \glmpt \gnot{\hspace{-0,34cm}\tilde\beta^{-1}} \grmpb \gnl
\gvac{1} \gmu \gwcm{3} \gnl
\gvac{1} \gcn{2}{1}{2}{2} \gelt{\hspace{0,14cm}f^{-1}} \gvac{1} \gcl{1} \gnl 
\gvac{1} \gob{2}{F\s'} \gvac{2} \gob{1}{F}
\gend=
\gbeg{4}{7}
\got{2}{F\s'} \got{3}{F\s'} \gnl
\gcmu \gmp{g} \gcl{1} \gnl
\gmp{f} \gcn{1}{1}{1}{2} \gmu \gnl
\gvac{2} \hspace{-0,22cm} \glmptb \gnot{\hspace{-0,34cm}\lambda_{\tilde\beta}'} \grmptb \gnl
\gvac{1} \hspace{-0,24cm} \gelt{\hspace{0,14cm}g^{-1}} \gcn{1}{1}{2}{1} \gcmu \gnl
\gvac{1} \gmu \gelt{\hspace{0,14cm}f^{-1}} \gcl{1} \gnl
\gvac{1} \gob{1}{\hspace{0,32cm}F\s'} \gvac{2} \gob{1}{F\s'} \gnl
\gend\stackrel{\equref{lambda beta calculo}}{=}
\gbeg{4}{7}
\got{2}{F\s'} \got{1}{F\s'} \gnl
\gcmu  \gcl{2} \gnl
\gmp{f} \gcl{1} \gnl
\gvac{1} \glmptb \gnot{\lambda_{\tilde\beta\s*gf^{-1}}'} \gcmptb \grmp \gnl
\gelt{\hspace{0,14cm}g^{-1}} \gcl{1} \gcl{2} \gnl
\gmu \gnl
\gob{2}{\hspace{0,32cm}F\s'} \gob{1}{F\s'} \gnl
\gend\stackrel{\equref{lambda alfa calculo}}{=}\lambda_{g^{-1}f*\tilde\beta\s*gf^{-1}}'=\lambda_\alpha'
$$
It remains to see that given $(Y,\psi_Y,\phi_Y)\in\Bimnd(F, F\s')$ the object $\G\left((Y,\psi_Y,\phi_Y)\right)=(\dcr Y,\psi_{\dcr Y},\phi_{\dcr Y})$ 
satisfies condition \equref{twisted YD cond}. We find: 
$$
\gbeg{3}{5}
\got{1}{F\s'} \got{1}{\dcr Y} \got{1}{F} \gnl
\glmptb \gnot{\hspace{-0,34cm}\psi_{\dcr Y}} \grmptb \gcl{1} \gnl
\gcl{1} \glmptb \gnot{\hspace{-0,34cm}\lambda_\beta} \grmptb \gnl
\glmptb \gnot{\hspace{-0,34cm}\phi_{\dcr Y}} \grmptb \gcl{1} \gnl
\gob{1}{F\s'} \gob{1}{\dcr Y} \gob{1}{F}
\gend=
\gbeg{6}{9}
 \got{2}{F\s'} \got{1}{Y} \got{3}{F} \gnl
\gcmu \gcl{2} \gvac{1} \gcl{4} \gnl
\gmp{f} \gbmp{\tilde\beta} \gnl
\gvac{1} \glmptb \gnot{\hspace{-0,34cm}\psi_Y} \grmptb \gnl 
\gvac{1} \gcl{1} \glmpt \gnot{\hspace{-0,34cm}\beta^{-1}} \grmpb \gnl
\gvac{1} \gcn{2}{1}{1}{3} \glmptb \gnot{\hspace{-0,34cm}\lambda_\beta} \grmptb \gnl
\gelt{g^{-1}} \gvac{1} \glmptb \gnot{\hspace{-0,34cm}\phi_Y} \grmptb \gcl{2} \gnl
\gwmu{3} \gcl{1} \gnl
\gvac{1} \gob{1}{F\s'} \gvac{1} \gob{1}{Y} \gob{1}{F} 
\gend\stackrel{\equref{l6}}{=}
\gbeg{5}{9}
\got{2}{F\s'} \got{1}{Y} \got{1}{F} \gnl
\gcmu \gcl{2} \gcl{3} \gnl
\gmp{f} \gbmp{\tilde\beta} \gnl
\gvac{1} \glmptb \gnot{\hspace{-0,34cm}\psi_Y} \grmptb \gnl 
\gvac{1} \gcl{2} \glmptb \gnot{\hspace{-0,34cm}\lambda} \grmptb \gnl
\gvac{2} \gcl{1} \glmptb \gnot{\hspace{-0,34cm}\beta^{-1}} \grmp \gnl
\gelt{g^{-1}} \glmptb \gnot{\hspace{-0,34cm}\phi_Y} \grmptb \gcl{2} \gnl
\gmu \gcl{1} \gnl
\gob{2}{F\s'} \gob{1}{Y} \gob{1}{F} 
\gend\stackrel{YD}{=}
\gbeg{5}{8}
 \got{2}{F\s'} \got{1}{Y} \got{1}{F} \gnl
\gcmu \gcl{1} \gcl{1} \gnl
\gmp{f} \gbmp{\tilde\beta} \glmptb \gnot{\hspace{-0,34cm}\phi_Y} \grmptb \gnl
\gvac{1} \glmptb \gnot{\hspace{-0,34cm}\lambda'} \grmptb \gcl{1} \gnl 
\gvac{1} \gcl{2} \glmptb \gnot{\hspace{-0,34cm}\psi_Y} \grmptb \gnl
\gelt{g^{-1}} \gvac{1} \gcl{1} \glmptb \gnot{\hspace{-0,34cm}\beta^{-1}} \grmp \gnl
\gmu \gcl{1} \gcl{1} \gnl
\gob{2}{F\s'} \gob{1}{Y} \gob{1}{F} 
\gend\stackrel{\equref{l6}}{=}
\gbeg{5}{7}
 \got{2}{F\s'} \got{1}{Y} \got{1}{F} \gnl
\gcmu \glmptb \gnot{\hspace{-0,34cm}\phi_Y} \grmptb \gnl
\gmp{f} \glmptb \gnot{\hspace{-0,34cm}\lambda'_{\tilde\beta}} \grmptb \gcl{2} \gnl 
\gvac{1} \gcl{2} \gbmp{\tilde\beta} \gnl 
\gelt{g^{-1}} \gvac{1} \glmptb \gnot{\hspace{-0,34cm}\psi_Y} \grmptb \gnl
\gmu \gcl{1} \glmptb \gnot{\hspace{-0,34cm}\beta^{-1}} \grmp \gnl
\gob{2}{F\s'} \gob{1}{Y} \gob{1}{F} 
\gend
$$

$$
\stackrel{\equref{lambda alfa calculo}}{=}
\gbeg{5}{9}
\got{1}{F\s'} \gvac{1} \got{2}{Y} \got{0}{F} \gnl
\gcl{2}  \gvac{1} \hspace{-0,34cm} \gelt{\hspace{0,14cm}g^{-1}} \glmptb \gnot{\hspace{-0,34cm}\phi_Y} \grmptb \gnl
\gvac{2} \gmu \gcl{4} \gnl
\gvac{1} \hspace{-0,3cm} \glmptb \gnot{\lambda_{\alpha}'} \gcmp \grmptb \gnl
\gvac{1} \gcl{4} \gvac{1} \hspace{-0,34cm} \gcmu \gnl
\gvac{3} \gmp{f} \gbmp{\tilde\beta} \gnl
\gvac{4} \glmptb \gnot{\hspace{-0,34cm}\psi_Y} \grmptb \gnl
\gvac{4} \gcl{1} \glmptb \gnot{\hspace{-0,34cm}\beta^{-1}} \grmp \gnl
\gvac{1} \gob{1}{\hspace{0,32cm}F\s'} \gvac{2} \gob{1}{Y} \gob{1}{F} \gnl
\gend=
\gbeg{3}{5}
\gvac{1} \got{1}{F\s'} \got{1}{\dcr Y} \got{1}{F} \gnl
\gvac{1} \gcl{1} \glmptb \gnot{\hspace{-0,34cm}\phi_{\dcr Y}} \grmptb \gnl
\gvac{1} \glmptb \gnot{\hspace{-0,34cm}\lambda'} \grmptb \gcl{1} \gnl
\gvac{1} \gcl{1} \glmptb \gnot{\hspace{-0,34cm}\psi_{\dcr Y}} \grmptb \gnl
\gvac{1} \gob{1}{F\s'} \gob{1}{\dcr Y} \gob{1}{F}
\gend
$$
\qed\end{proof}

In view of \leref{obv pair} we clearly have:

\begin{cor}
For every bimonad $F$ in $\K$ and every $\alpha\in G$ there is an isomorphism of categories $\Bimnd^T_{(\alpha,\alpha)}(F, F)\iso\Bimnd(F, F).$
\end{cor}

The proof of the following result is direct.

\begin{prop}
Let $(\A,F), (\A',F\s')$ be bimonads in $\K$ with isomorphic 0-automorphism groups $G$ and $G'$, respectively. Let $X$ be a 1-cell in $\K$ so that there is a 1-cell 
$(X, \tau_{F,X}):(\A,F)\to (\A',F\s')$ in $\Mnd(\K)$ and a 1-cell $(X, \tau_{X,F}):(\A,F)\to (\A',F\s')$ in $\Comnd(\K)$, so that $\tau_{X,F}\comp\tau_{F,X}=if_{FX}$ 
and 
\begin{equation} \eqlabel{YBE}
\gbeg{3}{5}
\got{1}{F'} \got{1}{X} \got{1}{F} \gnl
\gcl{1} \glmptb \gnot{\hspace{-0,34cm}\tau_{X,F}} \grmptb \gnl
\glmptb \gnot{\hspace{-0,34cm}\tau_{F',F'}} \grmptb \gcl{1} \gnl
\gcl{1} \glmptb \gnot{\hspace{-0,34cm}\tau_{F,X}} \grmptb \gnl
\gob{1}{F'} \gob{1}{X} \gob{1}{F}
\gend=
\gbeg{3}{5}
\got{1}{F'} \got{1}{X} \got{1}{F} \gnl
\glmptb \gnot{\hspace{-0,34cm}\tau_{F,X}} \grmptb \gcl{1} \gnl
\gcl{1} \glmptb \gnot{\hspace{-0,34cm}\tau_{F,F}} \grmptb \gnl
\glmptb \gnot{\hspace{-0,34cm}\tau_{X,F}} \grmptb \gcl{1} \gnl
\gob{1}{F'} \gob{1}{X} \gob{1}{F}
\gend
\end{equation} 
hold. Take $\alpha\in G', \beta\in G$ and suppose there is a pair in involution $(f,g)$ corresponding to $(\alpha, \beta)$ so that the following conditions hold: 
\begin{center}
\begin{tabular} {p{4.5cm}p{0cm}p{4.3cm}p{0cm}p{4.3cm}} 
\begin{equation} \eqlabel{nat beta}
\gbeg{3}{4}
\got{1}{F'} \got{1}{X} \gnl
\gbmp{\tilde\beta} \gcl{1} \gnl
\glmptb \gnot{\hspace{-0,34cm}\tau_{F,X}} \grmptb \gnl
\gob{1}{X} \gob{1}{F}
\gend=
\gbeg{3}{4}
\got{1}{F'} \got{1}{X} \gnl
\glmptb \gnot{\hspace{-0,34cm}\tau_{F,X}} \grmptb \gnl
\gcl{1} \gbmp{\tilde\beta} \gnl
\gob{1}{X} \gob{1}{F}
\gend
\end{equation} & & \vspace{-0,6cm}
\begin{equation} \eqlabel{nat g inv}
\gbeg{3}{5}
\got{1}{F'} \gnl
\gcl{1} \gelt{\hspace{0,14cm}g^{-1}} \gnl
\gcl{1} \gcl{1} \gnl
\glmptb \gnot{\hspace{-0,34cm}\tau_{F\s',F\s'}} \grmptb \gnl
\gob{1}{F\s'} \gob{1}{F\s'}
\gend=
\gbeg{3}{5}
\got{3}{F'} \gnl
\gvac{1} \gcl{3} \gnl
\gelt{\hspace{0,14cm}g^{-1}} \gnl
\gcl{1} \gnl
\gob{1}{F\s'} \gob{1}{F\s'}
\gend
\end{equation} & & 
\begin{equation} \eqlabel{nat f}
\gbeg{3}{4}
\got{1}{F'} \got{1}{F'} \gnl
\glmptb \gnot{\hspace{-0,34cm}\tau_{F\s',F\s'}} \grmptb \gnl
\gcl{1} \gmp{f} \gnl
\gob{1}{F\s'} 
\gend=
\gbeg{3}{4}
\got{1}{F'} \got{1}{F'} \gnl
\gmp{f} \gcl{2} \gnl
\gob{3}{F\s'} 
\gend
\end{equation} 
\end{tabular} 
\end{center}
Then $(X, \psi, \phi)\in\Bimnd^T(F_\beta, F\s'_\alpha)$ where $\psi$ and $\phi$ are given by 
$$
\psi=
\gbeg{4}{4}
\got{2}{F\s'} \got{1}{X} \gnl
\gcmu \gcl{1} \gnl
\gmp{f} \glmptb \gnot{\hspace{-0,34cm}\tau_{F,X}} \grmptb \gnl
\gvac{1} \gob{1}{X} \gob{1}{F} 
\gend \hspace{2,5cm}
\phi=
\gbeg{4}{4}
\gvac{1} \got{1}{X} \got{1}{F} \gnl
\gelt{g^{-1}} \glmptb \gnot{\hspace{-0,34cm}\tau_{X,F}} \grmptb \gnl
\gmu \gcl{1} \gnl
\gob{2}{F\s'} \gob{1}{X.}
\gend
$$
\end{prop}

The above Theorem, Corollary and Proposition generalize to the 2-categorical setting Theorem 4.1, Corollary 4.2 and Example 2.7, respectively, from \cite{PS}.


\end{document}